\documentclass{lmcs}
\pdfoutput=1

\usepackage{lastpage}
\lmcsdoi{18}{3}{1}
\lmcsheading{}{\pageref{LastPage}}{}{}%
{Dec.~29,~2020}{Jul.~14,~2022}{}

\usepackage[utf8]{inputenc}
\usepackage{xspace}
\usepackage{array}
\usepackage{amssymb}
\usepackage{mathabx}
\usepackage{stmaryrd}
\usepackage{tikz-cd}
  \usetikzlibrary{arrows,arrows.meta}
  \tikzcdset{arrow style=tikz, diagrams={>=stealth'}}
\usepackage{multicol}
\usepackage{booktabs}
\usepackage[mathcal]{euscript}
\usepackage[scr=boondox]{mathalfa}
\usepackage{stackengine}
\usepackage{scalerel}
\usepackage{float}

  \DeclareFontFamily{U}{mathc}{}
  \DeclareFontShape{U}{mathc}{m}{it}%
  {<->s*[1.03] mathc10}{}
  \DeclareMathAlphabet{\mathfun}{U}{mathc}{m}{it}

  \newenvironment{subproof}[1][\proofname]{
  \begin{proof}[#1]
    }{
  \end{proof}
}

\newcommand{\ms}[1]{\mathscr{#1}}
\newcommand{\topo}[1]{\mathbb{#1}}  
\newcommand{\cat}[1]{\mathsf{#1}}   
\newcommand{\fun}[1]{\mathfun{#1}}  
\newcommand{\mo}[1]{\mathcal{#1}}   
\newcommand{\lan}[1]{\boldsymbol{\mathcal{#1}}}   

\newenvironment{deduction}{%
\par\smallskip\noindent\begin{tabular}{>{(}wl{2ex}<{)}>{$}wl{8mm}<{$}>{$}wl{65mm}<{$}wr{4cm}}}{\end{tabular}\par\smallskip\noindent}

\newenvironment{wideded}{%
\par\smallskip\noindent\begin{tabular}{>{(}wl{2ex}<{)}>{$}wl{24mm}<{$}>{$}wl{65mm}<{$}wr{4cm}}}{\end{tabular}\par\smallskip\noindent}

\usepackage{array,tabularx}

\renewcommand{\iff}{\quad\text{iff}\quad}
\renewcommand{\tilde}[1]{\widetilde{#1}}

\renewcommand{\phi}{\varphi}
\renewcommand{\epsilon}{\varepsilon}
\renewcommand{\hat}[1]{\widehat{#1}}

\DeclareMathOperator{\id}{id}
\DeclareMathOperator{\op}{op}
\DeclareMathOperator{\Prop}{Prop}
\DeclareMathOperator{\Var}{Var}

\mathchardef\hyphen="2D

\newcommand{\wwedge}{\mathrel{\stackon[-1.5pt]{$\wedge$}{$\wedge$}}}
\newcommand{\mon}{{\vartriangle}} 
\newcommand{\oth}{\operatorname{th}}
\DeclareMathOperator{\Ax}{Ax}

\newcommand{\llb}{\llbracket}
\newcommand{\rrb}{\rrbracket}

\renewcommand{\Box}{{\boxempty}}
\newcommand{\dbox}{{\boxdot}}
\newcommand{\ddiamond}{{\mathbin{\rotatebox[origin=c]{45}{$\boxdot$}}}}

\newcommand{\mto}{\xrightarrow{\circ}}

\DeclareSymbolFont{symbolsC}{U}{txsyc}{m}{n}
\DeclareMathSymbol{\co}{\mathrel}{symbolsC}{128}

\newcommand{\CFil}{\fun{Clp^f}}
\newcommand{\TFil}{\fun{F^t}}

\newcommand{\MI}{\mathbf{MI}}
\newcommand{\Int}{\fun{Int}}

\newcommand{\entails}{\vdash}
\newcommand{\Hilb}{\ms{H}}
\newcommand{\Eq}{\ms{E}}
\newcommand{\mopo}{\textrm{mp}\xspace}
\newcommand{\assum}{\ensuremath{\mathrm{ass}}}
\newcommand{\axio}{\ensuremath{\mathrm{ax}}}
\newcommand{\refl}{\ensuremath{\mathrm{ref}}}

\begin{document}

\title[Modal meet-implication logic]{Modal meet-implication logic}

\author[J.~de Groot]{Jim de Groot}	

\author[D.~Pattinson]{Dirk Pattinson}	
\address{The Australian National University, Canberra, Australia}	
\email{\{\texttt{jim.degroot}, \texttt{dirk.pattinson}\}\texttt{@anu.edu.au}}  

\begin{abstract}
  We extend the meet-implication fragment of propositional intuitionistic logic
  with a meet-preserving modality. We give semantics based on semilattices
  and a duality result with a suitable notion of descriptive frame.
  As a consequence we obtain completeness and identify a common (modal) fragment
  of a large class of modal intuitionistic logics.
  
  We recognise this logic as a dialgebraic logic, and as a consequence obtain
  expressivity-somewhere-else.
  Within the dialgebraic framework, we then investigate the extension of
  the meet-implication fragment of propositional intuitionistic logic with a
  monotone modality and prove completeness and expressivity-somewhere-else for it.
\end{abstract}

\maketitle

\section{Introduction}

  In the field of mathematical logic, one does not only study logical formalisms,
  but also the relationships between them.
  The main reason for this is that one can often transfer results from one logic
  to the other, like complexity or completeness results. 

  Two formalisms may differ because they have a different set of connectives.
  For example, classical propositional logic contains connectives that are not
  expressible in positive logic.
  Nonetheless, the two logics are closely related:
  classical propositional logic is a conservative extension of positive logic.
  That is, a positive propositional formula is valid in positive logic
  if and only if it is classically valid, when viewed as a formula of
  classical propositional logic.
  Similarly intuitionistic logic is a conservative
  extension of its $(\top, \wedge, \to)$-fragment (that we call $\mathbf{MI}$)
  and of its positive fragment.
%
  Even if two logical formalisms use the same language,
  they can still have different axioms or rules.
  An example of this phenomenon can be found in the multitude of different modal extensions
  of intuitionistic logic \cite{Bul66,Fis80,BozDos84,Yok85,PloSti86,WolZak98}.

  Another interesting instance is the comparison of positive \emph{modal} logic
  and classical modal logic. This traces back to Dunn \cite{Dun95},
  whose paper on positive modal logic was motivated by finding an
  axiomatisation for the positive fragment of normal modal logic.
  As a consequence, the logic he found automatically has classical normal modal logic
  as a conservative extension.
  More recently, the monotone modal logic analogue hereof has been studied \cite{Gro20}.

  In analogy to this, we investigate the relation between modal extensions
  of intuitionistic logic and modal extensions of $\mathbf{MI}$.
  On first sight this may seem like a daunting task, considering the wide
  variety of modal intuitionistic logics mentioned above.
  However, it turns out that the
  extension $\mathbf{MI_{\Box}}$ of $\mathbf{MI}$ with a meet-preserving
  unary modality $\Box$ is a common denominator of many of them,
  thus making the task at hand much more manageable.
  More precisely, we prove that the modal extensions of intuitionistic logic given by
  Bo\v{z}i\'{c} and Do\v{s}en \cite{BozDos84}, Plotkin and Sterling \cite{PloSti86},
  Fisher Servi \cite{Fis80},
  and Wolter and Zakharyaschev \cite{WolZak98} are all conservative extensions of
  $\mathbf{MI_{\Box}}$.

  In order to establish this, we carry out a semantic study of
  $\mathbf{MI}$ and $\mathbf{MI_{\Box}}$.
  We begin by giving semantics for $\mathbf{MI}$
  based on meet-semilattices, that we call I-models.
  In such I-models, formulae are interpreted as filters;
  these form natural denotations of formulae because the collection of filters of
  a meet-semilattice is closed under intersections but not
  under unions.
  We restrict the known duality between meet-semilattices and so-called
  M-spaces from \cite{DavWer83} to a duality for \emph{implicative} semilattices.
  We refer to those M-spaces dual to an implicative semilattice as I-spaces.
  They form a category that is isomorphic to a category of descriptive frames.
  (That this situation is completely analogous to the intuitionistic logic
  framework becomes apparent when looking at Figure \ref{fig:int}.)

  Subsequently, we enrich the language of $\mathbf{MI}$ with a box-like
  modal operator that preserves finite meets to obtain $\mathbf{MI_{\Box}}$.
  This modal operator is interpreted via its own accessibility relation
  in what we call $\Box$-frames.
  Based on this, we define \emph{descriptive} $\Box$-frames for $\mathbf{MI_{\Box}}$.
  These give rise to a duality for $\mathbf{MI_{\Box}}$ which 
  piggy-backs on the duality for implicative semilattices. 
  As a consequence of this duality we not only derive weak completeness,
  but we also observe that many famous modal \emph{intuitionistic} logics
  are conservative extensions of~$\mathbf{MI_{\Box}}$.
  
  We then recognise the setup for $\mathbf{MI_{\Box}}$ as an instance of a
  \emph{dialgebraic logic}. The framework of dialgebraic logic is a slight
  generalisation of coalgebraic logic and was recently proposed as a
  suitable framework for the study of modal intuitionistic logics \cite{GroPat20}.
  After we detail how $\mathbf{MI_{\Box}}$ fits the dialgebraic perspective,
  we instantiate results from {\it op.~\!cit.}~to obtain the notion of
  a \emph{filter extension} of a frame for $\mathbf{MI_{\Box}}$.
  Furthermore, we obtain \emph{expressivity-somewhere-else}:
  two states in two models are logically equivalent if and only if
  certain corresponding states in the filter extensions of the models
  are behaviourally equivalent.

  Once we have this dialgebraic perspective on modal extensions of
  $\mathbf{MI}$, we use it to sketch the extension $\mathbf{MI_{\mon}}$ of $\mathbf{MI}$ with a
  monotone modal operator $\mon$. As a consequence of the dialgebraic approach,
  we get weak completeness, filter extensions, and expressivity-somewhere-else.
  Furthermore, with only minor additional effort one can see that
  intuitionistic logic with a \emph{geometric} modality,
  see e.g.~\cite[Section 6]{Gol93}, is a conservative extension of $\mathbf{MI_{\mon}}$.

\subsection*{Structure of the paper}
  In Section \ref{sec:sl} we recall basic facts about meet-semilattices,
  a duality for meet-semilattices, and implicative semilattices.
  Subsequently, in Section \ref{sec:base} we give a semantics for the
  meet-implication fragment of intuitionistic logic that is based on meet-semilattices.
  We also define the notion of a descriptive I-frame, show that these are
  isomorphic to so-called I-spaces, which in turn are dually equivalent to
  implicative semilattices.
  
  In Section \ref{sec:normal} we start the study of a normal modal extension
  of $\mathbf{MI}$. We provide frame semantics by means of what we call $\Box$-frames,
  and show how to view these as dialgebras.
  Building on this, in Section \ref{sec:duality} we define general and descriptive
  $\Box$-frames. We then show that these can again be viewed as dialgebras,
  and we prove a duality with the algebraic semantics of $\mathbf{MI_{\Box}}$.
  In Section \ref{sec:fragments} we prove that many modal intuitionistic logics
  are conservative extensions of $\mathbf{MI_{\Box}}$.
  
  Subsequently, in Section \ref{sec:dialg}, we recall the framework of
  dialgebraic logic and show that $\mathbf{MI_{\Box}}$ is an instance
  of this. As a consequence, we get a notion of filter extensions
  and expressivity-somewhere-else.
  In Section \ref{sec:monotone} we make use of this dialgebraic perspective to
  investigate $\mathbf{MI}$ with a monotone modality $\mon$, and obtain similar results
  as for $\Box$.

\subsection*{Related work}
%
  The meet-implication fragment of intuitionistic propositional logic has been studied by itself in
  \cite[Section 4.C]{Cur63} and more recently in e.g.~\cite{BezEA15,FonMor14}.
  Its appeal partly lies in the fact that the category of meet-semilattices
  is finitely generated, whereas the category of Heyting algebras is not.
  Implicative meet-semilattices have also been studied by themselves,
  see e.g.~\cite{Nem65,NemWha71}.
  In \cite{Mon55} a set of equational axioms for implicative meet-semilattices
  is given, which proves that the category of implicative meet-semilattices
  is a variety of algebras.
  
  There are several different dualities for (implicative) meet-semilattices.
  In \cite{Bal69}, the author gives a representation theorem for prime semilattices and implicative
  semilattices by means of prime filters.
  A duality for meet-semilattices where the dual spaces are topologised meet-semilattices
  was given by Davey and Werner in \cite{DavWer83}.
  More recently, in \cite{BezJan11} and \cite{BezJan13} a Priestley and Esakia-style duality for distributive
  (implicative) meet-semilattices is developed.
  We have opted for the restriction of the duality from \cite{DavWer83}
  because it arises as a topologised version of implicative semilattices and,
  as such, is intuitive to use as a notion of descriptive frame.


\section{Meet-semilattices}\label{sec:sl}

  We recall basic definitions and facts that are either known or elementary.

\subsection*{Posets}
  Let $(X, \leq)$ be a poset and $a \subseteq X$. We write
  ${\uparrow}a = \{ x \in X \mid y \leq x \text{ for some } y \in a \}$
  and we call $a$ \emph{up-closed} or an \emph{upset} if ${\uparrow}a = a$.
  Similarly define ${\downarrow}a$ and \emph{downsets}.
  A \emph{monotone function} from $(X, \leq)$ to $(X', \leq')$ is a map
  $f : X \to X'$ such that $x \leq y$ implies $f(x) \leq' f(y)$.
  It is called \emph{bounded} or a \emph{p-morphism} if moreover for all
  $x \in X$ and $y' \in X'$ such that $f(x) \leq' y'$ there exists
  $y \in X$ such that $x \leq y$ and $f(y) = y'$.
  We write $\cat{Pos}$ and $\cat{Krip}$ for the categories of posets with
  monotone morphisms and bounded morphisms, respectively.

\subsection*{Meet-semilattices}
  A \emph{meet-semilattice} is a poset in which every finite subset has
  a greatest lower bound, called the \emph{meet}.
  The meet two elements $a$ and $b$ is denoted by
  $a \wedge b$, and the meet of the empty set is the top element $\top$.
  A \emph{meet-semilattice morphism} is a monotone
  function that preserves finite meets (hence also $\top$).
  We sometimes write $(X, \wedge, \top)$ for the meet-semilattice corresponding 
  corresponding to a poset $(X, \leq)$.
  We can recover $\leq$ from $(X, \wedge, \top)$ via $x \leq y$ iff $x \wedge y = x$.
  
  The category of meet-semilattices (with top)
  and meet-semilattice morphisms is denoted by $\cat{SL}$.
  It forms a variety of algebras, and the free
  meet-semilattice on a set $X$ is given by finite subsets of $X$
  (intuitively, these are finite meets), with $\wedge$ given by union
  and top element $\emptyset$.

  Throughout this text, by \emph{semilattice} we always mean
  meet-semilattice with top.

\subsection*{The contravariant filter functor}
  A \emph{filter} in a semilattice $A$ is a non-empty up-closed subset that is
  closed under finite meets.
  A \emph{prime down-set} is a down-closed subset $d$ of $A$
  that satisfies: if $a \wedge a' \in d$ then $a \in d$ or $a' \in d$.
  A subset $p \subseteq A$ is a filter of $A$ if and only if its complement
  $p^c = A \setminus p$ is a prime downset.
  Let $2$ be the two-element semilattice and $p \subseteq A$.
  Then the characteristic map $\chi_p : A \to 2$, which sends the elements of $p$
  to $1$ and all other elements to $0$, is a semilattice morphism
  if and only if $p$ is a filter.

  We write $\fun{F}A$ be the collection of filters of a semilattice $A$.
  With conjunction $\cap$ and top element $A$, the set $\fun{F}A$ forms a semilattice again.
  We can extend this assignment to a contravariant functor $\fun{F} : \cat{SL} \to \cat{SL}$
  by defining the action of $\fun{F}$ on a morphism $h : A \to A'$
  by $\fun{F}h : \fun{F}A' \to \fun{F}A : p' \mapsto h^{-1}(p')$.


  
  For every $a \in A$ the collection
  $\tilde{a} = \{ p \in \fun{F}A \mid a \in p \}$ is a filter in $\fun{F}A$.
  This gives rise to a morphism $\eta_A : A \to \fun{FF}A : a \mapsto \tilde{a}$.
  In fact, the collection
  $\eta = (\eta_A)_{A \in \cat{SL}} : \id_{\cat{SL}} \to \fun{FF}$ is a natural transformation.
  To see this, suppose $h : A \to B$ is a homomorphism, and let $a \in A$
  and $p \in \fun{FF}B$. Then we have
  $$
    p \in (\fun{FF}h)(\eta_A(a))
      \iff f^{-1}(p) \in \eta_A(a)
      \iff a \in f^{-1}(p)
      \iff f(a) \in p
      \iff p \in \eta_B(f(a)).
  $$

\begin{prop}
  The functor $\fun{F} : \cat{SL} \to \cat{SL}$ is dually adjoint to
  itself. It defines a dual adjunction between $\cat{SL}$ and itself
  with units $\eta_A$.
\end{prop}
\begin{proof}
  We verify that $\fun{F}\eta_A \circ \eta_{\fun{F}A} = \id_{\fun{F}A}$.
  Let $a \in A$ and $p \in \fun{F}A$. Then we have
  $$
    a \in \fun{F}\eta_A(\eta_{\fun{F}A}(p))
      \iff \eta_A(a) \in \eta_{\fun{F}A}(p)
      \iff p \in \eta_A(a)
      \iff a \in p.
  $$
  This completes the proof.
\end{proof}

\subsection*{Duality for semilattices}
  We can adapt the dual adjunction to a dual equivalence by topologising one side of the
  dual adjunction.
  An \emph{M-space} is a tuple
  $\topo{X} = (X, \wedge, \top, \tau)$ such that $(X, \wedge, \top)$ is a meet-semilattice,
  and $(X, \tau)$ is a Stone space generated by a subbase
  of clopen filters and their complements. 
  We write $\cat{MSpace}$ for the category of M-spaces and
  continuous semilattice homomorphisms.

\begin{lem}\label{lem:closed-clp-fil}
  Let $\topo{X} = (X, \wedge, \top, \tau)$ be an M-space
  and $c \subseteq X$ a filter in $(X, \wedge, \top)$.
  Then $c$ is closed in $(X, \tau)$ if and only if
  it is the intersection of all clopen filters that contain $c$.
\end{lem}
\begin{proof}
  The direction from right to left follows immediately from the fact that
  the arbitrary intersection of clopen sets is closed.
  For the converse, suppose $x \notin c$.
  Then for each $y \in c$ there exists a clopen filter $a_y$
  such that $y \in a_y$ and $x \notin a_y$.
  (To see this, note that $x \wedge y \leq y$. Since
  $(X, \tau)$ is a Stone space there exists a clopen set $a$ containing
  $y$ and not $x \wedge y$. Since $(X, \tau)$ is generated by clopen filters
  and their complements, we may assume that $a$ is either a clopen filter
  or the complement of a clopen filter. But the latter is not an option,
  as it would imply $x \wedge y \in a$, so $a$ is a clopen filter.
  Furthermore, we have $x \notin a$ because $x \in a$ would imply
  $x \wedge y \in a$, as $a$ is a filter.)
  Then we have an open cover $c \subseteq \bigcup \{ a_y \mid y \in c \}$,
  so there exists a finite subcover, say, $c \subseteq a_1 \cup \cdots \cup a_n$.
  Finally, we claim $c \subseteq a_i$ for some $1 \leq i \leq n$.
  For suppose this is not the case, then for each $a_i$ we can find an element
  $z_i$ such that $z_i \in c$ and $z_i \notin a_i$. This implies
  $z_1 \wedge \cdots \wedge z_n \in c$, because $c$ is a filter,
  but $z_1 \wedge \cdots \wedge z_n$ not in $a_i$ for each $i$,
  because that would imply $z_i \in a_i$, a contradiction.
  So there must exist a clopen filter containing $c$. Moreover
  this clopen filter does not contain $x$ by construction.
  It follows that $c$ is the intersection of all clopen filters containing it.
\end{proof}

  Define $\TFil : \cat{SL} \to \cat{MSpace}$ by sending a semilattice $A$
  to the semilattice of filters with a topology generated by
  $$
    \tilde{a} = \{ p \in \fun{F}A \mid a \in p \},
    \qquad
    \tilde{a}^c = \{ p \in \fun{F}A \mid a \notin p \},
  $$
  where $a$ ranges over $A$. For a morphism $h$ define $\TFil h = h^{-1}$.
  Conversely, define $\CFil : \cat{MSpace} \to \cat{SL}$ by sending an
  M-space $\topo{X}$ to its collection of clopen filters,
  viewed as a semilattice with meet $\cap$ and top $\topo{X}$. 
  For a continuous semilattice morphism $h$ set $\CFil h = h^{-1}$.
  The following result is originally due to Hofmann, Mislove and Stralka
  \cite{HofMisStr74}, and can also be found in \cite{DavWer83}.
  
\begin{thm}\label{thm:msl-duality}
  The functors $\TFil$ and $\CFil$ define a dual equivalence
  $\cat{MSpace} \equiv^{\op} \cat{SL}$.
\end{thm}



\subsection*{Distributive semilattices}
  A meet-semilattice $A$ is called \emph{distributive} if for all
  $a, b, c \in A$, whenever $a \wedge b \leq c$ there exist
  $a', b'$ such that $a \leq a', b \leq b'$ and $c = a' \wedge b'$.
  Write $\cat{DSL}$ for the full subcategory of $\cat{SL}$ with
  distributive semilattices as objects.

  In a distributive meet-semilattice $A$ we can define the smallest filter
  containing two given filters $p$ and $q$ by
  $$
    \langle p, q \rangle = \{ a \wedge b \mid a \in p, b \in q \}.
  $$
  To see that this is again a filter, note that 
  $\top \in \langle p, q \rangle$ because $\top = \top \wedge \top$
  and $\top \in p, \top \in q$;
  the set $\langle p, q \rangle$ is closed under meets because
  $p$ and $q$ are, so
  $(a \wedge b) \wedge (a' \wedge b') = (a \wedge a') \wedge (b \wedge b')$;
  and it is up-closed precisely by distributivity.
  
  The following proposition also follows
  from~\cite[Section II.5, Lemma~1]{Gra78}.

\begin{prop}
  If $A$ is a distributive semilattice, then so is $\fun{F}A$.
\end{prop}
\begin{proof}
  Let $p, q, r \in \fun{F}A$ be such that $p \cap q \subseteq r$.
  Let $p' = \langle p, r \rangle$ and $q' = \langle q, r \rangle$.
  We claim that $p' \cap q' = r$.
  For the inclusion from right to left, observe
  $$
    r = (p \cap q) \cup r = (p \cup r) \cap (q \cup r) \subseteq p' \cap q'.
  $$
  For the converse we need to work a bit harder.
  Suppose $x \in p' \cap q'$. Since $x \in p'$ we can write $x = y \wedge z$
  for some $y \in p$ and $z \in r$.
  Moreover, $x \leq y$ (because $x \wedge y = (y \wedge z) \wedge y = y \wedge z = x$),
  so $y \in p' \cap q'$. Using that $y \in q'$ we write $y = y' \wedge z'$
  for some $y' \in q$ and $z' \in r$.
  Since $y \in p$ and $y \leq y'$ we must have $y' \in p$, so that
  $y' \in p \cap q$. In particular this means $y' \in r$.
  We now have $x = y \wedge z = y' \wedge z' \wedge z$,
  where $y', z', z \in r$. This implies $x \in r$, and hence
  $p' \cap q' \subseteq r$.
\end{proof}

\begin{cor}
  The dual adjunction between $\cat{SL}$ and itself restricts to
  a dual adjunction between $\cat{DSL}$ and itself.
\end{cor}

\subsection*{Implicative semilattices}
  A semilattice $A$ is \emph{implicative} if we can define a binary
  operation $\mto$ such that $x \leq y \mto z$ iff $x \wedge y \leq z$
  for all $x, y, z \in A$.
  An implicative semilattice homomorphism is a map that preserves
  top, meets and implications.
  We write $\cat{ISL}$ for the category of implicative semilattice and
  homomorphisms.

\begin{prop}\label{prop:dist-imp}
  If $A$ is a distributive semilattice, then $\fun{F}A$ is an implicative semilattice.
\end{prop}
\begin{proof}
  Let $a, b \in \fun{F}A$.
  We prove that $C = \{ x \in X \mid \text{ if } x \leq y
                            \text{ and } y \in a
                           \text{ then } y \in b \}$ is again a filter.
  Clearly it is up-closed and contains $\top$. Suppose $x, y \in C$ and let
  $z$ be such that $x \wedge y \leq z$ and $z \in a$.
  Since $A$ is assumed to be distributive, we can find $x', y'$ such that
  $x \leq x'$, $y \leq y'$ and $z = x' \wedge y'$.
  We have $x' \in a$, because $z \leq x'$
  (indeed, $z \wedge x' = x' \wedge y' \wedge x' = x' \wedge y' = z$)
  and for the same reason $y' \in a$. But then we must have $x', y' \in b$,
  because $x \leq x'$ and $y \leq y'$, and since $b$ is a filter
  $z = x' \wedge y' \in b$. Therefore $C$ is a filter of $A$,
  so $C = a \mto b$.
\end{proof}

  Since the semilattice underlying an implicative semilattice is
  automatically distributive, this implies:

\begin{cor}\label{cor:imp-imp}
  If $A$ is an implicative semilattice, then so is $\fun{F}A$.
\end{cor}

\section{Frames for the meet-implication-fragment of intuitionistic logic}\label{sec:base}

  In this section we describe semantics and a duality for the base logic
  underlying the modal logics discussed in this paper: the meet-implication
  fragment of intuitionistic logic.
  The language $\lan{L}$ of this fragment is given by the grammar
  $$
    \phi ::= \top \mid p \mid \phi \wedge \phi \mid \phi \to \phi
  $$
  where $p$ ranges over $\Prop$, some set of proposition letters.
  The algebraic semantics of the fragment are given by implicative semilattices. 
  They have been well-researched, see e.g.~\cite{Nem65}.
  We give a Hilbert-style axiomatisation of this logic.

\begin{defi}
  The Hilbert-style deductive system for the meet-implication
  fragment of intuitionistic propositional logic is given by the
  axioms
  \begin{enumerate}[(H$_1$)]
    \item $p \to (q \to p)$
    \item $(p \to (q \to r)) \to ((p \to q) \to (p \to r))$
    \item $(p \wedge q) \to p$
    \item $(p \wedge q) \to q$
    \item $p \to (q \to (p \wedge q))$
    \item $\top$
  \end{enumerate}
  If $\Gamma$ is a set of formulae and $\phi$ is a formula, we 
  say that $\phi$ \emph{is deducible in $\Hilb$} from $\Gamma$ 
  if $\Gamma \entails_\Hilb \phi$ can be derived using the
  rules 
  \[ \mbox{(\assum)}\frac{}{\Gamma \entails_\Hilb \phi} (\mbox{if } \phi \in \Gamma)
  \qquad\quad
     \mbox{(\axio)}\frac{}{\Gamma \entails_\Hilb \phi} (\mbox{if } \phi \in H) 
  \qquad\quad
   \mbox{(\mopo)}\frac{\Gamma \entails_\Hilb \phi \qquad \Gamma \entails_\Hilb
  \phi \to \psi}{\Gamma \entails_\Hilb \psi}
  \]
  where $H$ is the set of substitution instances of the axioms
  \ref{it:H1} to \ref{it:H6} above.
  We write $\Hilb \entails \phi$ if $\emptyset \entails_\Hilb \phi$
  and we let $\bf{MI}$ denote the collection of formulae $\phi$ such that
  $\Hilb \entails \phi$.
\end{defi}

  It is proved in the appendix that this Hilbert style axiomatisation
  indeed corresponds to the equational logic of implicative semilattices.
  
  The fact that the meet-implication reduct of a Heyting algebra is an
  implicative semilattice implies that any sound semantics for
  intuitionistic logic also provides a sound semantics for $\mathbf{MI}$.
  The best-known frame semantics of intuitionistic logic is given by posets,
  often called intuitionistic Kripke frames.
  So intuitionistic Kripke frames can be taken as frame semantics for $\mathbf{MI}$.

  
  However, we will use a more restricted version of intuitionistic Kripke frames
  as our semantics.
  The main advantages of this restriction are that it lies closer to the duality
  between frame semantics and implicative semilattices we intend to use.
  Moreover, its completeness implies completeness with respect to intuitionistic
  Kripke frames in a straightforward way.
  We exploit the fact that the collection of filters of an implicative semilattice
  is again an implicative semilattice and use these as frame semantics
  for $\lan{L}$-formulae.
  Since every (implicative) semilattice carries an underlying partial order
  given by $x \leq y$ iff $x \wedge y = x$, we can think of the frame semantics
  as intuitionistic Kripke frames with extra properties.
  When using them as frame semantics, we refer to implicative semilattices
  as I-frames, both to distinguish them from the algebraic semantics,
  and because morphisms between I-frames are not the same as implicative
  semilattice homomorphisms.

\begin{defi}
  An \emph{M-frame} a semilattice viewed as intuitionistic Kripke frame.
  That is, an M-frame is is a poset $(X, \leq)$ with meets.
  We write $\cat{MF}$ for the category of M-frames and functions that
  preserve finite meets. (Indeed, $\cat{MF} \cong \cat{SL}$.)

  An \emph{I-frame} is an implicative semilattice regarded as a Kripke frame.
  An I-frame morphism from $(X, \leq)$ to $(X', \leq')$ is an $M$-frame 
  morphism $f$ such that $f(x) \leq' y'$ implies the existence of a $y \in X$
  such that $x \leq y$ and $f(y) = y'$.
  We write $\cat{IF}$ for the category of I-frames and I-frame morphisms.
\end{defi}

  As announced, I-frames are of course simply implicative semilattices,
  viewed as a special kind of intuitionistic Kripke frame.
  Besides, note that an I-frame morphism is an M-frame morphism that
  is bounded with respect to the underlying posets.
  We now define I-models as I-frames with a valuation.
  Rather than using upsets as denotations of formulae in our language, we use
  \emph{filters} (i.e.~upsets closed under finite meets).
%
  This is natural because the collection of filters
  is closed under intersections, but not under unions.


\begin{defi}\label{def:I-model}
  A \emph{valuation} for an I-frame $(X, \leq)$
  is a function $V$ that assigns to each proposition letter $p$
  a filter of $(X, \leq)$, and an \emph{I-model} is an I-frame together with a valuation.
  An \emph{I-model morphism} from $(X, \leq, V)$ to $(X', \leq', V')$ is
  an I-frame morphism $f : (X, \leq) \to (X', \leq')$ satisfying $V = f^{-1} \circ V'$.
  
  The interpretation of $\lan{L}$-formulae in an I-model $\mo{M} = (X, \leq, V)$
  is given recursively via
  \begin{align*}
    \mo{M}, x \Vdash \top
      &\quad\text{always} \\
    \mo{M}, x \Vdash p
      &\iff x \in V(p) \\
    \mo{M}, x \Vdash \phi \wedge \psi
      &\iff \mo{M}, x \Vdash \phi \text{ and } \mo{M}, x \Vdash \psi \\
    \mo{M}, x \Vdash \phi \to \psi
      &\iff \forall y \in X, \text{ if } x \leq y \text{ and } \mo{M}, y \Vdash \phi \text{ then } \mo{M}, y \Vdash \psi
  \end{align*}


  Let $\mo{M} = (X, \leq, V)$ be an I-model.
  We write $\mo{M} \Vdash \phi$ if $\mo{M}, x \Vdash \phi$
  for all $x \in X$. For an I-frame $(X, \leq)$ we write
  $(X, \leq) \Vdash \phi$ if $\mo{M} \Vdash \phi$ for every
  I-model based $\mo{M}$ on $(X, \leq)$.
\end{defi}

  Clearly every I-model is also an intuitionistic Kripke model,
  and formulae are interpreted in the same way as in intuitionistic Kripke models.
  Furthermore, an I-model morphism is in particular a bounded morphism between
  Kripke models. Since these preserve truth of intuitionistic formulae 
  we have:

\begin{prop}
  Let $h : (X, \leq, V) \to (X', \leq', V')$ be an I-model morphism.
  Then for all $x \in X$ and $\phi \in \lan{L}$ we have
  $$
    x \Vdash \phi \iff f(x) \Vdash \phi.
  $$
\end{prop}



  Towards a duality result between certain (general) I-frames and implicative semilattices,
  we now define \emph{descriptive} I-frames and corresponding morphisms.
  Apart from giving us completeness for $\mathbf{MI}$ with respect to I-frames,
  the duality underlies dualities for modal extensions in subsequent sections.

\begin{defi}
  A \emph{descriptive I-frame} is a tuple $(X, \leq, A)$
  consisting of an I-frame $(X, \leq)$ and a collection
  $A \subseteq \fun{F}(X, \leq)$ of filters such that
  \begin{enumerate}[$({I}_1)$]
    \item \label{it:IM1}
          $A$ is closed under $\cap$ and $\mto$, given by
          \begin{equation}\label{eq:mto}
            a \mto b
              = \{ x \in X \mid 
                   \text{for all } y \geq x,
                   \text{ if } y \in a
                   \text{ then } y \in b \};
          \end{equation}
    \item $A$ is differentiated: if $x \not\leq y$ in $(X, \leq)$ then there
          exists $a \in A$ such that $x \in a$ and $y \notin a$;
    \item $(X, \leq, A)$ is compact: every cover of $X$
          consisting of elements in $A$ and complements of elements in $A$
          has a finite subcover.
  \end{enumerate}
  A \emph{descriptive I-frame morphism} from $(X, \leq, A)$ to $(X', \leq', A')$
  is an I-frame morphism $h : (X, \leq) \to (X', \leq')$ such that
  $h^{-1}(a') \in A$ for all $a' \in A'$.
  Write $\cat{D\hyphen IF}$ for the category of descriptive I-frames and
  descriptive I-frame morphisms.
\end{defi}

  Alternatively, we can view a descriptive I-frame as a special kind of M-space.
  We shall use this in order to derive a duality between descriptive I-frames and implicative
  semilattices:
  First we define the subcategory $\cat{ISpace}$ of $\cat{MSpace}$ (Definition \ref{def:ISpace});
  second we prove that the category of descriptive I-frames and its morphisms is
  isomorphic to $\cat{ISpace}$ (Proposition \ref{prop:DIF-IS});
  and finally we show that the duality between M-spaces and
  semilattices from Theorem \ref{thm:msl-duality} restricts to a duality
  between I-spaces and \emph{implicative} semilattices (Theorem \ref{thm:IS-ISL}).
  Diagrammatically it looks as in the left diagram of Figure \ref{fig:int}.
  To illustrate that the situation is completely analogous to that of intuitionistic logic,
  we have depicted the isomorphisms and dualities for intuitionistic logics in the
  diagram on the right, where the horizontal dualities in the right column are known as Priestley
  duality and Esakia duality.
  
  \begin{figure}[H]
  $$
    \begin{tikzcd}[row sep=large,
                   column sep=large,
                   every matrix/.append style={draw, inner ysep=7pt, inner xsep=6pt, rounded corners}]
        & \cat{MSpace}
            \arrow[r, -, "\equiv^{\op}" below, "\text{Thm.~\ref{thm:msl-duality}}" above]
        & \cat{SL} \\
      \cat{D\hyphen IF}
            \arrow[r, -, "\cong" below, "\text{Prop.~\ref{prop:DIF-IS}}" above]
        & \cat{ISpace}
            \arrow[u, >->]
            \arrow[r, -, "\equiv^{\op}" below, "\text{Thm.~\ref{thm:IS-ISL}}" above]
        & \cat{ISL} \arrow[u, >->]
    \end{tikzcd}
    \qquad
    \begin{tikzcd}[row sep=large,
                   column sep=large,
                   every matrix/.append style={draw, inner ysep=7pt, inner xsep=6pt, rounded corners}]
        & \cat{Pries}
            \arrow[r, -, "\equiv^{\op}"]
        & \cat{DL} \\
      \cat{D\hyphen Krip}
            \arrow[r, -, "\cong"]
        & \cat{Esakia}
            \arrow[u, >->]
            \arrow[r, -, "\equiv^{\op}"]
        & \cat{HA} \arrow[u, >->]
    \end{tikzcd}
  $$
  \caption{Dualities for intuitionistic logic and its meet-implication fragment.}
  \label{fig:int}
  \end{figure}
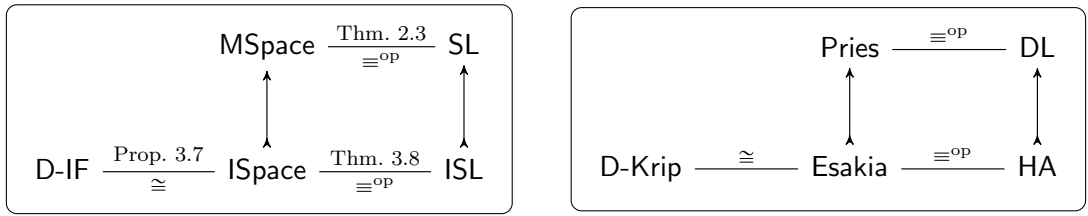

\begin{defi}\label{def:ISpace}
  An \emph{I-space} is an M-space whose collection of clopen filters is
  closed under the operation $\mto$ from Equation \eqref{eq:mto}.
  An I-space morphism from $\topo{X} \to \topo{Y}$ is a bounded M-space morphisms.
  We write $\cat{ISpace}$ for the category of I-spaces and I-space morphisms.
\end{defi}

\begin{prop}\label{prop:DIF-IS}
  The category $\cat{D\hyphen IF}$ is isomorphic to $\cat{ISpace}$.
\end{prop}
\begin{proof}[Proof sketch]
  We only give the constructions.
  If $\topo{X}$ is an I-space with underlying set $X$ we can define an order
  $\leq$ on $X$ by $x \leq y$ if $x \wedge y = x$.
  Then the tuple $(X, \leq, \CFil\topo{X})$ is a descriptive I-frame.
  
  Conversely, for a descriptive I-frame $(X, \leq, A)$ we already know that
  $(X, \leq)$ is a poset with a top element and binary meets, because it is
  an I-frame, so we have a semilattice $(X, \wedge, \top)$.
  Generate a topology $\tau_A$ from the subbase $A \cup -A$, where
  $-A = \{ X \setminus a \mid a \in A \}$. Then $\topo{X} = (X, \wedge, \top, \tau_A)$
  is an I-space.
\end{proof}


  Proposition \ref{prop:DIF-IS} allows us to use the notions of
  a descriptive I-frame and an I-space interchangeably,
  and henceforth we will do so.
  We now prove that the category if I-spaces is dually equivalent to
  the category of implicative semilattices. This ultimately proves that
  descriptive I-frames are dual to implicative semilattices.

\begin{thm}\label{thm:IS-ISL}
  We have a dual equivalence $\cat{ISpace} \equiv^{\op} \cat{ISL}$.
\end{thm}
\begin{proof}
  Let $\topo{X}$ be an I-space.
  Since in particular it is an M-space we know from Theorem \ref{thm:msl-duality}
  that $\CFil\topo{X}$ is a semilattice with meets given by intersections.
  It is easy to show that $\mto$, defined as in~\eqref{eq:mto}, is residuated
  with respect to $\cap$ in $\CFil\topo{X}$, so in fact the latter is implicative.
  In the converse direction we have:

\begin{clm}
  If $A$ is an implicative semilattice, then $\TFil A$ is an $I$-space.
\end{clm}
\begin{subproof}[Proof of claim]
  By Theorem \ref{thm:msl-duality} $\TFil A$ is an M-space
  and it follows from Corollary \ref{cor:imp-imp} that the semilattice
  underlying $\TFil A$ is implicative, hence an I-frame. So we are left to show that
  the clopen filters are closed under $\mto$.
  We do this by proving that $\tilde{a} \mto \tilde{b} = \widetilde{a \to b}$
  for all clopen filters $\tilde{a}, \tilde{b} \in \CFil(\TFil A)$,
  where $\to$ is the implication from $A$ and $\mto$ is defined on
  the filters of $\TFil A$ as in~\eqref{eq:mto}.
  
  Suppose $p \in \widetilde{a \to b}$ and $p \subseteq q$,
  where $p, q \in \TFil A$. If $q \in \tilde{a}$, then
  we have $a \to b \in q$ and $a \in q$, hence
  $(a \to b) \wedge a \in q$. Since $(a \to b) \wedge a \leq b$ this
  implies $b \in q$, so $q \in \tilde{b}$.
  Therefore we have $p \in \tilde{a} \mto \tilde{b}$.
  
  Conversely, suppose $p \notin \widetilde{a \to b}$.
  Then $a \to b \notin p$ and $b \notin p$ because $b \leq a \to b$.
  We construct a filter $q \supseteq p$ such that $a \in q$ and $b \notin q$,
  thus witnessing $p \notin ({\downarrow}(\tilde{a} \cap \tilde{b}^c))^c = \tilde{a} \mto \tilde{b}$.
  Define $q = {\uparrow} \{ d \wedge a \mid d \in p \}$.
  This is up-closed by definition and one can easily see that $q$ is closed
  under finite meets, so $q$ is a filter. Suppose towards a contradiction that
  $a \to b \in q$, then there is $d \in p$ such that
  $d \wedge a \leq a \to b$, so $d \wedge a = d \wedge a \wedge a \leq b$
  so $d \leq a \to b$. This implies $a \to b \in p$, which is a contradiction.
  So $p \notin \widetilde{a \to b}$.
\end{subproof}

  So we have now established the duality on objects.
  We proceed to show that the restriction of $\CFil$ to the morphisms of $\cat{ISpace}$
  lands in $\cat{ISL}$.

\begin{clm}
  Let $f : \topo{X} \to \topo{X}'$ be an I-space morphism.
  Then $\CFil f : \CFil\topo{X}' \to \CFil\topo{X}$ is an
  implicative semilattice homomorphism.
\end{clm}
\begin{subproof}[Proof of claim]
  We already know that $\CFil f$ is a semilattice morphism
  because of the dual equivalence between $\cat{SL}$ and $\cat{MSpace}$,
  so we only need to show that it preserves implications.
  Let $a', b' \in \CFil\topo{X}'$. We show that
  $f^{-1}(a' \mto b') = f^{-1}(a') \mto f^{-1}(b')$.
  
  Suppose $x \in f^{-1}(a' \mto b')$. Let $y$ be any successor of $x$.
  If $y \in f^{-1}(a')$ then $f(y) \in a'$.
  Since $f(x) \in a' \mto b'$ and $f(x) \leq' f(y)$ this implies $f(y) \in b'$, so
  $y \in f^{-1}(b')$. Thus every successor $y$ of $x$ satisfies
  $y \notin f^{-1}(a')$ or $y \in f^{-1}(b')$, that is,
  $x \in f^{-1}(a') \mto f^{-1}(b')$.
  
  Conversely, suppose $x \in f^{-1}(a') \mto f^{-1}(b')$.
  We aim to show that $f(x) \in a' \mto b'$.
  Suppose $f(x) \leq' y'$ and $y' \in a'$.
  Then since $f$ is bounded there
  is $y \in X$ such that $x \leq y$ and $f(y) = y'$. By assumption
  $y \in f^{-1}(a')$ hence $y \in f^{-1}(b')$. This implies
  $y' = f(y) \in b'$, so that $f(x) \in a' \mto b'$, i.e.~%
  $x \in f^{-1}(a' \mto b')$.
\end{subproof}

  Conversely, applying $\TFil$ to a semilattice homomorphism $h$
  yields a continuous semilattice homomorphism.
  We show that if moreover $h$ preserves implication,
  then $\TFil h$ is bounded.
  The proof is a basic adaptation of a similar proof for 
  Esakia spaces and Heyting algebras.
  
  We make use of the following facts: If $A$ is a meet-semilattice then
  the space $\TFil A$ has a basis of clopens of the form
  $\tilde{a} \cap \tilde{b}_1^c \cap \cdots \cap \tilde{b}_m^c$,
  where $m$ is a non-negative integer and $a, b_1, \ldots, b_m \in A$.
  Furthermore, for any filter $p \in \TFil A$ we have
  $$
    \{ p \} = \bigcap \{ D \mid D \text{ is a basic clopen and } p \in D \}.
  $$

\begin{clm}
  Let $h : A \to A'$ be an implicative semilattice morphism.
  Then $\TFil h : \TFil A' \to \TFil A$ is bounded.
\end{clm}
\begin{subproof}[Proof of claim]
  Let $q' \in \TFil A'$ and $p \in \TFil A$ be such that
  $\TFil h(q') \subseteq p$.
  Let $C = \tilde{a} \cap \tilde{b}_1^c \cap \cdots \cap \tilde{b}_m^c$
  be a basic clopen containing $p$. Then
  \begin{align*}
    (\TFil h)^{-1}({\downarrow}c)
      &= (\TFil h)^{-1}({\downarrow}\tilde{a} \cap \tilde{b}_1^c \cap \cdots \cap \tilde{b}_m^c) \\
      &= (\TFil h)^{-1}(({\downarrow}\tilde{a} \cap \tilde{b}_1^c)
                    \cap ({\downarrow}\tilde{a} \cap \tilde{b}_m^c)) \\
      &= (\TFil h)^{-1}({\downarrow}\tilde{a} \cap \tilde{b}_1^c) \cap \cdots \cap
         (\TFil h)^{-1}({\downarrow}\tilde{a} \cap \tilde{b}_m^c) \\
      &= (\TFil h)^{-1}((\widetilde{a \to b_1})^c) \cap \cdots \cap
         (\TFil h)^{-1}((\widetilde{a \to b_m})^c) \\
      &= \widetilde{h(a \to b_1)}^c \cap \cdots \cap \widetilde{h(a \to b_m)}^c \\
      &= \widetilde{(h(a) \to h(b_1))}^c
         \cap \cdots \cap \widetilde{(h(a) \to h(b_m))}^c \\
      &= ({\downarrow}\tilde{h(a)} \cap \tilde{h(b_1)}^c) \cap \cdots \cap
         ({\downarrow}\tilde{h(a)} \cap \tilde{h(b_m)}^c) \\
      &= {\downarrow}\tilde{h(a)} \cap \tilde{h(b_1)}^c \cap \cdots \cap \tilde{h(b_m)}^c \\
      &= {\downarrow}(\TFil h)(\tilde{a}) \cap (\TFil h)(\tilde{b}_1)^c
                                           \cap \cdots \cap (\TFil h)(\tilde{b}_m)^c \\
      &= {\downarrow}((\TFil h)^{-1}(c))
  \end{align*}
  Since $\TFil h(q') \subseteq p$ we have
  $q' \in (\TFil h)^{-1}({\downarrow}C) = {\downarrow}((\TFil h)^{-1}(C))$
  for every basic clopen $C$ containing $p$.
  This implies ${\uparrow}q' \cap (\TFil h)^{-1}(C) \neq \emptyset$.

  Since the finite intersection of basic clopens containing $p$ is nonempty,
  compactness implies that
  $\bigcap \{ {\uparrow}q' \cap (\TFil h)^{-1}(C) \} \neq \emptyset$,
  where $C$ ranges over the basic clopens containing $p$.
  Since the intersection of basic clopens containing $p$ is $\{ p \}$
  we have
  \begin{equation*}
    \begin{split}
      \emptyset
        \neq \bigcap \big\{ {\uparrow}q' \cap (\TFil h)^{-1}(C) \big\}
        &= {\uparrow}q' \cap \bigcap (\TFil h)^{-1}(C) \\
        &= {\uparrow}q' \cap (\TFil h)^{-1}\Big(\bigcap C \Big)
        = {\uparrow}q' \cap (\TFil h)^{-1}(\{ p \})
    \end{split}
  \end{equation*}
  so there must exist $p' \in \fun{F}A'$ such that $q' \subseteq p'$ and
  $\TFil h(p') = p$. This proves that $\TFil h$ is bounded.
\end{subproof}

  In conclusion, the duality from Theorem \ref{thm:msl-duality} restricts to a
  duality between $\cat{ISL}$ and $\cat{MSpace}$.
\end{proof}
  
  Given that $\MI$ is sound and complete with respect to implicative
  semilattices by construction, this theorem implies completeness with
  respect to descriptive I-frames, where propositional variables are
  interpreted as admissible filters.

\begin{thm}\label{thm:MI-complete}
  The system $\MI$ is complete with respect to the classes of
  \begin{enumerate}
    \item descriptive I-frames;
    \item I-frames;
    \item intuitionistic Kripke frames.
  \end{enumerate}
\end{thm}
\begin{proof}
  For the first item,
  if $\MI \not \vdash \phi$ then there is an algebraic model (i.e.~an
  implicative semilattice with a valuation) where $\phi$ does not evaluate to $\top$.
  But this implies that the dual I-space model does not satisfy $\phi$.
  
  The second and third item follow since, in particular,
  I-space models are I-models, and I-models are
  intuitionistic Kripke models
\end{proof}

  Finally, we note that $\MI$ is \emph{not} the same as the $(\top, \wedge, \to)$-fragment
  of classical propositional logic. This is witnessed by the fact that
  Peirce's formula
  $$
    ((p \to q) \to p) \to p
  $$
  is valid classically but not intuitionistically. Therefore, as a consequence
  of Theorem \ref{thm:MI-complete} it is not valid in $\MI$.

\section{A normal modal extension}\label{sec:normal}

  We enrich the meet-implication fragment $\mathbf{MI}$ of intuitionistic logic
  with a modal operator $\Box$ satisfying
  \begin{enumerate}[($\Box_1$)]
    \item\label{ax:box1} $\Box(p \wedge q) \leftrightarrow \Box p \wedge \Box q$; and
    \item\label{ax:box2} $\Box\top \leftrightarrow \top$.
  \end{enumerate}
  Write $\lan{L}_{\Box}$ for the extension of $\lan{L}$ with unary operator $\Box$,
  and $\mathbf{MI_{\Box}}$ for the extension of (the derivation system giving) $\mathbf{MI}$ with the
  axioms \ref{ax:box1} and \ref{ax:box2} and the congruence rule
  $$
    \mbox{(cong)}\frac{p \leftrightarrow q}{\Box p \leftrightarrow \Box q}.
  $$
  The algebras corresponding to $\mathbf{MI_{\Box}}$
  are given by \emph{implicative semilattices with operators} (ISLOs).
  These are pairs $(A, \Box)$ consisting of an implicative semilattice $A$ and a
  meet-preserving function $\Box : A \to A$.
  A homomorphism between ISLOs $(A, \Box)$ and $(A', \Box')$ is an implicative
  semilattice homomorphism $h : A \to A'$ that additionally satisfies
  $h \circ \Box = \Box' \circ h$.
  Let $\cat{ISL_{\Box}}$ denote the category of ISLOs and their morphisms.
  
  The main goal of this section is twofold: First, we give relational semantics
  for $\cat{MI_{\Box}}$, in the form of I-frames with an additional relation.
  We turn such frames into a category by defining an appropriate notion of
  morphism, which we show to be truth-preserving. Furthermore, we prove that
  these relational structures provide a sound semantics for $\mathbf{MI_{\Box}}$.
  Second, we show that both our category of relational frames, as well as 
  the category $\cat{ISL_{\Box}}$, are categories of \emph{dialgebras}.
  This realisation plays a key r\^{o}le in proving a duality result
  for $\cat{MI_{\Box}}$ in Section \ref{sec:duality}, which in turn
  underlies the completeness result for $\mathbf{MI_{\Box}}$ with respect to our
  frame semantics.
  
  Let us begin by defining $\Box$-frames.

\begin{defi}\label{def:box-frm}
  A $\Box$-frame is a tuple $(X, \leq, R)$ where $(X, \leq)$ is an I-frame
  and $R$ is a binary relation on $X$ satisfying:
  \begin{enumerate}[($B_1$)]
    \item \label{it:box-frm1}
          $\top R x$ iff $x = \top$, and $xR\top$ for all $x$;
    \item \label{it:box-frm2}
          If $xRy \leq z$ then $xRz$;
    \item \label{it:box-frm3}
          If $xRy$ and $x'Ry'$ then $(x \wedge x')R(y \wedge y')$;
    \item \label{it:box-frm4}
          If $(x \wedge x')Rz$ then there are $y, y' \in X$ such that
          $xRy$ and $x'Ry'$ and $y \wedge y' = z$.
  \end{enumerate}
  A $\Box$-model is a $\Box$-frame together with a valuation $V : \Prop \to \fun{F}A$
  that assigns a filter to each proposition letter.
  The interpretation of a formula $\phi$ in a $\Box$-model $\mo{M} = (X, \leq, R, V)$
  is given recursively by the clauses of Definition \ref{def:I-model} and
  $$
    \mo{M}, x \Vdash \Box\phi \iff xRy \text{ implies } \mo{M}, y \Vdash \phi.
  $$
\end{defi}

\begin{defi}
  A \emph{$\Box$-frame morphism} from $(X, \leq, R)$ to $(X', \leq', R')$
  is an I-frame morphism $h : (X, \leq) \to (X', \leq')$ between the
  underlying I-frames such that $h : (X, R) \to (X', R')$ is a bounded morphism,
  i.e.~$xRy$ implies $h(x)R'h(y)$ and for all $x \in X$ and $y' \in X'$, if $f(x)R'y'$ then there exists
  $y \in X$ such that $xRy$ and $f(y) = y'$.
  A \emph{$\Box$-model morphism} is a $\Box$-frame morphism that is also
  an I-model morphism between the underlying I-models.

  We denote the category of $\Box$-frames and $\Box$-frame morphisms
  by $\cat{IF_{\Box}}$.
\end{defi}

  An easy induction on the structure of a formula shows that:

\begin{prop}
  If $h : \mo{M} \to \mo{M}'$ is a $\Box$-model morphism,
  $x$ is a world in $\mo{M}$ and $\phi$ is an $\mathbf{MI_{\Box}}$-morphism,
  then
  $$
    \mo{M}, x \Vdash \phi \iff \mo{M}', h(x) \Vdash \phi.
  $$
\end{prop}

\begin{prop}[Soundness]
  If $\mathbf{MI_{\Box}} \vdash \phi$ then $\cat{IF_{\Box}} \Vdash \phi$.
\end{prop}
\begin{proof}
  We need to prove that every frame satisfies
  $\Box(\phi \wedge \psi) \leftrightarrow \Box\phi \wedge \Box\psi$ and
  $\Box\top \leftrightarrow \top$. Both follow immediately from the definition.
\end{proof}

  We can view both the category $\cat{IF_{\Box}}$ of $\Box$-frames as well as
  the category $\cat{ISL_{\Box}}$ of implicative semilattices with modal operators
  as categories of dialgebras. Dialgebras generalise both algebras and coalgebras,
  and were introduced by Hagino in~\cite{Hag87} to describe data types.
  They were subsequently investigated in~\cite{PolZwa01,AltEA11,Blo12,Cia13}.
  We recall their definition:

\begin{defi}
  Let $\fun{G}, \fun{H} : \cat{C} \to \cat{D}$ be two functors.
  A $(\fun{G}, \fun{H})$-dialgebra is a pair $(X, \gamma)$ consisting of
  an object $X \in \cat{C}$ and a morphism $\gamma : \fun{G}X \to \fun{H}X$ in $\cat{D}$.
  A dialgebra morphism from $(X, \gamma)$ to $(X', \gamma')$ is a morphism
  $f : X \to X'$ in $\cat{C}$ such that $\fun{H}f \circ \gamma = \gamma' \circ \fun{G}f$.
  The collection of $(\fun{G}, \fun{H})$-dialgebras and morphisms constitutes
  the category $\cat{Dialg}(\fun{G}, \fun{H})$. 
\end{defi}

  It is fairly straightforward to model ISLOs as dialgebras.
  Let $\fun{N} : \cat{ISL} \to \cat{SL}$ be the functor that sends an
  implicative meet-semilattice $A$ to the free meet-semilattice generated by
  $\{ \Box a \mid a \in A \}$ modulo $\Box\top = \top$ and 
  $\Box a \wedge \Box b = \Box(a \wedge b)$,
  and a homomorphism $h : A \to A'$ to $\fun{N}h : \fun{N}A \to \fun{N}A'$
  generated by $\fun{N}h(\Box a) = \Box h(a)$. Then we have:

\begin{prop}\label{prop:ISLb-dialg}
  Let $\fun{j} : \cat{ISL} \to \cat{SL}$ be the inclusion functor,
  then
  $$
    \cat{ISL_{\Box}} \cong \cat{Dialg}(\fun{N}, \fun{j}).
  $$
\end{prop}

  We leave the obvious proof of this proposition to the reader.
  Observe that $\fun{N}$ is naturally isomorphic to $\fun{j}$,
  so that $\cat{Dialg}(\fun{N}, \fun{j}) \cong \cat{Dialg}(\fun{j}, \fun{j})$.
  We shall exploit this fact in Section \ref{sec:duality}
  to obtain a duality between ISLOs and descriptive $\Box$-frames.

  Modelling $\Box$-frames as dialgebras requires a bit more work.
  Crucially, this uses the inclusion functor $\fun{i} : \cat{IF} \to \cat{MF}$
  and the \emph{covariant filter functor}
  $\fun{G} : \cat{IF} \to \cat{MF}$. This functor acts the same on objects as
  the functor $\fun{F}$, but differs in its action on morphisms.
  (Think of the co- and contravariant powerset functors in coalgebraic logic over
  a classical base, where the covariant powerset functor gives rise to
  Kripke frames~\cite{Mos99}, and the covariant powerset functor takes a set to its
  algebra of predicates.)
  We define this functor and prove that it is well defined.
  Subsequently, in Theorem \ref{thm:iG-dialg} we show that the category of $\Box$-frames
  is isomorphic to $\cat{Dialg}(\fun{i}, \fun{G})$.

\begin{defi}\label{def:fun-G}
  Let $(X, \leq)$ be an I-frame. Define $\fun{G}(X, \leq)$ to be the collection
  of filters ordered by reverse inclusion.
  This is a meet-semilattice with top element $\{ \top \}$ and meet defined by
  $$
    a \wwedge b := \langle a, b \rangle = \{ x \wedge y \mid x \in a, y \in b \}.
  $$
  For an $\cat{IF}$-morphism $h : (X, \leq) \to (X', \leq')$ define
  $$
    \fun{G}h : \fun{G}(X, \leq) \to \fun{G}(X', \leq')
             : a \mapsto h[a] = \{ h(x) \mid x \in a \}.
  $$
\end{defi}

  Note that $a \wwedge b$ was shown to be a filter in Section \ref{sec:sl}.
  In fact, it is the smallest filter containing both $a$ and $b$,
  so alternatively we can define
  $a \wwedge b = \bigcap \{ c \in \fun{F}(X, \leq) \mid a \cup b \subseteq c \}$.
  Thus $\fun{G}(X, \leq)$ has binary meets and a top element.
  Hence $\fun{G}$ is well defined on objects. We show that the same
  is true for morphisms.

\begin{lem}
  If $h$ is a $\cat{IF}$-morphism from $(X, \leq)$ to $(X', \leq')$,
  then $\fun{G}h$ is an $\cat{MF}$-morphism.
\end{lem}
\begin{proof}
  Let $a \in \fun{G}(X, \leq)$.
  Since $h$ is bounded the set $h[a]$ is up-closed in $(X', \leq')$.
  Moreover, it contains $\top'$ and is closed under $\wedge'$
  because $h$ preserves $\top$ and meets.
  One easily sees that $\fun{G}h(\{ \top \}) = \{ \top' \}$.
  Furthermore, for $a, b \in \fun{G}(X, \leq)$ we have
  \begin{align*}
    \fun{G}h(a) \wwedge' \fun{G}h(b)
      &= \{ x' \wedge' y' \mid x' \in h[a] \text{ and } y' \in h[b] \} \\
      &= \{ h(x) \wedge' h(y) \mid x \in a \text{ and } y \in b \} \\
      &= \{ h(x \wedge y) \mid x \in a \text{ and } y \in b \} \\
      &= \{ h(z) \mid z \in a \wwedge b \} \\
      &= \fun{G}h(a \wwedge b).
  \end{align*}
  So $\fun{G}h$ preserves the top element and binary meets, hence all finite meets.
\end{proof}

  A straightforward verification shows that $\fun{G}$ is indeed
  a functor. We proceed to the main theorem of this section.

\begin{thm}\label{thm:iG-dialg}
  Let $\fun{i} : \cat{IF} \to \cat{MF}$ be the inclusion functor.
  Then $\cat{IF_{\Box}} \cong \cat{Dialg}(\fun{i}, \fun{G})$.
\end{thm}

\begin{proof}
  We split the proof into three claims. The first two describe the isomorphism
  on objects, and the last one proves the isomorphism for morphisms.

\begin{clm}\label{clm:iG-dialg1}
  Let $(X, \leq, R)$ be a $\Box$-frame and define
  $\gamma_R : \fun{i}(X, \leq) \to \fun{G}(X, \leq) : x \mapsto R[x]$,
  where $R[x] := \{ y \in X \mid xRy \}$.
  Then $(X, \leq, \gamma)$ is a $(\fun{i}, \fun{G})$-dialgebra.
\end{clm}
\begin{subproof}[Proof of claim]
  For all $x \in X$ the set $\gamma_R(x)$ is a filter by definition of $R$.
  We verify that $\gamma_R$ preserves meets and top.
  By \ref{it:box-frm1} we have $\gamma_R(\top) = \{ \top \}$, which is the top element in
  $\fun{G}(X, \leq)$.
  It follows immediately from \ref{it:box-frm3} and \ref{it:box-frm4} that
  $\gamma_R(x) \wedge_{\fun{G}} \gamma_R(y) = \gamma_R(x \wedge y)$.
\end{subproof}

\begin{clm}\label{clm:iG-dialg2}
  Let $(X, \leq)$ be an I-frame and $(X, \leq, \gamma)$ a
  $(\fun{i}, \fun{G})$-dialgebra. Define $R_{\gamma}$ by
  $xR_{\gamma}y$ iff $y \in \gamma(x)$. Then $(X, \leq, R)$ is a $\Box$-frame.
\end{clm}
\begin{subproof}[Proof of claim]
  We need to verify that $R_{\gamma}$ satisfies the four items
  from Definition \ref{def:box-frm}.
  Since $\gamma(\top) = \{ \top \}$ \ref{it:box-frm1} holds.
  By definition $\gamma(x)$ is a filter in $(X, \leq)$, hence up-closed.
  This proves \ref{it:box-frm2}.
  Items \ref{it:box-frm3} and \ref{it:box-frm4} follow from the fact
  that $\gamma$ preserves meets.
\end{subproof}

  It is straightforward to see that Claims \ref{clm:iG-dialg1} and \ref{clm:iG-dialg2}
  establish an isomorphism on objects.
  We complete the proof of the theorem by showing that the $\Box$-frame morphisms between two $\Box$-frames
  are precisely the $(\fun{i}, \fun{G})$-dialgebra morphisms between the corresponding
  dialgebras.
  Let $(X, \leq, R)$ and $(X', \leq', R')$ be two $\Box$-frames
  with corresponding $(\fun{i}, \fun{G})$-dialgebras $(X, \leq, \gamma)$ and
  $(X', \leq', \gamma')$.

\begin{clm}
  An I-frame morphism $f : (X, \leq) \to (X', \leq')$ is a
  $\Box$-frame morphism from $(X, \leq, R)$ to $(X', \leq', R')$
  if and only if it is an $(\fun{i}, \fun{G})$-dialgebra morphism
  from $(X, \leq, \gamma)$ to $(X', \leq', \gamma')$.
\end{clm}
\begin{subproof}[Proof of claim]
  Suppose $f$ is a $\Box$-frame morphism. In order to prove that it is also
  a dialgebra morphism we need to show that
  $$
    \begin{tikzcd}
      \fun{i}(X, \leq)
            \arrow[r, "\fun{i}f"]
            \arrow[d, "\gamma"]
        & \fun{i}(X', \leq')
            \arrow[d, "\gamma'"] \\
      \fun{G}(X, \leq)
            \arrow[r, "\fun{G}f"]
        & \fun{G}(X', \leq')
    \end{tikzcd}
  $$
  commutes. To see this, let $x \in X$ and $y' \in Y'$. Then
  \begin{align*}
    y' \in \gamma'(\fun{i}f(x))
      &\iff f(x)R'y' \\
      &\iff \exists y \in X \text{ s.t. } xRy \text{ and } f(y) = y' \\
      &\iff \exists y \in X \text{ s.t. } y \in \gamma(x) \text{ and } f(y) = y' \\
      &\iff y' \in \fun{G}f(\gamma(x)).
  \end{align*}
  Conversely, suppose $f$ is a dialgebra morphism. Then $f$ is monotone with respect
  to $R$ and $R'$ because $xRy$ implies $y \in \gamma(x)$ hence
  $f(y) \in \fun{G}f(\gamma(x)) = \gamma'(\fun{i}f(x))$ so $f(x)R'f(y)$.
  For boundedness, suppose $f(x)R'y'$, then $y' \in \gamma'(\fun{i}f(x)) = \fun{G}f(\gamma(x))$,
  and this implies $f(y) = y'$ for some $y \in \gamma(x)$, i.e.~for some $y$ with $xRy$.
\end{subproof}
  
  Combining these three claims proves the theorem.
\end{proof}

\section{Descriptive \texorpdfstring{$\Box$}{box}-frames and duality}\label{sec:duality}

  Towards a duality for ISLOs
  we now define \emph{descriptive $\Box$-frames}.
  We then show that these are dialgebras for the functors
  $\fun{I}, \fun{V} : \cat{ISpace} \to \cat{MSpace}$, where $\fun{I}$ is the inclusion
  functor and $\fun{V}$ is a variation of the Vietoris functor.
  Subsequently, we prove that $\fun{V}$ is naturally isomorphic to $\fun{I}$,
  and as a consequence we obtain a duality between descriptive $\Box$-frames
  and implicative meet-semilattices with normal operators as follows:
  $$
    \cat{Dialg}(\fun{I}, \fun{V})
      \cong \cat{Dialg}(\fun{I}, \fun{I})
      \equiv^{\op} \cat{Dialg}(\fun{j}, \fun{j})
      \cong \cat{Dialg}(\fun{N}, \fun{j}).
  $$
  As a corollary we get completeness of $\mathbf{MI_{\Box}}$ with respect to
  $\Box$-frames.
  
\begin{defi}
  A \emph{general $\Box$-frame} is a tuple $(X, \leq, R, A)$ such that
  $(X, \leq, R)$ is a $\Box$-frame, $(X, \leq, A)$ is a descriptive I-frame,
  and $A$ is closed under
  $$
    m_{\Box} : \fun{F}(X, \leq) \to \fun{F}(X, \leq)
                    : p \mapsto \{ x \in X \mid R[x] \subseteq p \}.
  $$
  (Recall that we write $R[x] := \{ y \in X \mid xRy \}$.)
  It is called \emph{descriptive} if moreover
  $R[x] = \bigcap \{ a \in A \mid R[x] \subseteq a \}$ for all $x \in X$. 
  
  A morphism between general $\Box$-frames is a function which is
  simultaneously a $\Box$-frame morphism between the underlying $\Box$-frames,
  and a descriptive I-frame morphism between the underlying descriptive I-frames.
  The collection of descriptive $\Box$-frames and general $\Box$-frame morphisms
  forms a category, and since we tend to think of these as I-spaces
  with a collection of admissibles we will denote it by $\cat{ISpace_{\Box}}$.
\end{defi}

  As stated the category $\cat{ISpace_{\Box}}$ can be described as a category of
  dialgebras. We first define the relevant functor for doing so.

\begin{defi}\label{def:fun-V}
  Let $\topo{X} = (X, \leq, \tau)$ be an I-space. Call a filter
  $c \in \fun{F}(X, \leq)$ \emph{closed} if it is the intersection of all
  clopen filters in which it is contained. 
  (By Lemma~\ref{lem:closed-clp-fil} it is equivalent to require that
  $c$ is closed in $(X, \tau)$.)
  Define $\fun{V}\topo{X}$ to be the collection of 
  closed filters of $\topo{X}$ ordered by reverse inclusion and
  topologised by the subbase
  $$
    \dbox a = \{ c \in \fun{V}\topo{X} \mid c \subseteq a \},
    \qquad
    \ddiamond b = \{ c \in \fun{V}\topo{X} \mid c \cap b \neq \emptyset \},
  $$
  where $a$ ranges over the clopen filters of $\topo{X}$ and 
  $b$ over the clopen prime downsets of $\topo{X}$.
  For an I-space morphism $h : \topo{X} \to \topo{X}'$
  define $\fun{V}h = h[-]$, i.e.~$\fun{V}h$ sends a closed filter $p$
  to the direct image $h[p]$ of $p$ under $h$.
\end{defi}

  Note that the top element of $\fun{V}\topo{X}$ is $\{ \top \}$
  and conjunction of closed filters is given by
  $$
    c \wwedge c' = \bigcap \{ a \in A \mid c \cup c' \subseteq a \}.
  $$
  Before proving that descriptive $\Box$-frames are $(\fun{I}, \fun{V})$-dialgebras,
  we will prove that $\fun{V}$ is naturally isomorphic to $\fun{I}$.
  This entails the useful fact that $c \wwedge c' = \{ y \wedge z \mid y \in c, z \in c' \}$,
  which we can then use to prove $\cat{ISpace_{\Box}} \cong \cat{Dialg}(\fun{I}, \fun{V})$.

\begin{lem}\label{lem:V-iso}
  For every I-space $\topo{X}$ the M-space $\fun{V}\topo{X}$ is isomorphic
  to $\topo{X}$ in $\cat{MSpace}$. 
\end{lem}
\begin{proof}
  Since $\TFil$ and $\CFil$ constitute a duality, it suffices
  to prove that $\fun{V}\topo{X}$ is isomorphic to $\TFil(\CFil\topo{X})$.
  To achieve this, we define two continuous semilattice homomorphisms
  and show that these are each other's inverses.
  
  Define $f : \fun{V}\topo{X} \to \TFil (\CFil\topo{X})$ by
  $f(c) = \{ a \in \CFil(\topo{X}) \mid c \subseteq a \}$. It is easy to verify that
  this is a filter. Clearly $f(\{ \top \}) = \CFil(\topo{X})$, which is the
  top element of $\TFil(\CFil\topo{X})$, and
  \begin{align*}
    f(c \wwedge c')
      &= f\Big(\bigcap \big\{ a \in \CFil(\topo{X}) \mid c \cup c' \subseteq a \big\}\Big) \\
      &= f\Big(\bigcap \big\{ a \in \CFil(\topo{X}) \mid c \subseteq a \text{ and } c' \subseteq a \big\}\Big)
       = f(c) \cap f(c')
  \end{align*}
  so $f$ preserves meets. Furthermore,
  $$
    f^{-1}(\tilde{a})
      = \{ c \in \fun{V}\topo{X} \mid f(c) \in \tilde{a} \}
      = \{ c \in \fun{V}\topo{X} \mid a \in f(c) \}
      = \{ c \in \fun{V}\topo{X} \mid c \subseteq a \}
      = \dbox a
  $$
  so $f$ is continuous.
  
  In the converse direction, define
  $g : \TFil (\CFil\topo{X}) \to \fun{V}\topo{X}
     : p \mapsto \bigcap \{ a \in \CFil\topo{X} \mid a \in p \}$.
  This is well defined because $g(p)$ is the intersection of filters in $A$.
  Furthermore, $g(p) \subseteq a$ iff $a \in p$.
  (Proof: if $g(p) \subseteq a$ then there are (potentially infinite) $b_i \in A$
  such that $b_i \in p$ and $\bigcap b_i \subseteq a$. By compactness
  we can find a finite number of such $b_i$, say, $b_1, \ldots, b_m$,
  such that $b_1 \cap \cdots \cap b_m \subseteq a$. But then
  $b_1 \cap \cdots \cap b_m \in p$ because $p$ is a filter,
  hence $a \in p$ as $p$ is up-closed.)
  Moreover, $g(A) = \{ \top \}$ and
  \begin{align*}
    g(p \cap q)
      &= \bigcap \{ a \in A \mid a \in p \text{ and } a \in q \} \\
      &= \bigcap \{ a \in A \mid g(p) \subseteq a \text{ and } g(q) \subseteq a \} \\
      &= \bigcap \{ a \in A \mid g(p) \cup g(q) \subseteq a \} \\
      &= g(p) \wwedge g(q)
  \end{align*}
  so $g$ preserves top and binary meets, hence all finite meets.
  Finally, 
  $$
    g^{-1}(\dbox a) = \{ p \mid g(p) \subseteq a \} = \{ p \mid a \in p \} = \tilde{a}.
  $$
  proves continuity. Clearly $f$ and $g$ are inverses, and hence
  $\fun{V}\topo{X} \cong \TFil(\CFil\topo{X})$.
\end{proof}

  Since we know that $\topo{X}$ is isomorphic to $\TFil(\CFil\topo{X})$
  via the map $x \mapsto \tilde{x} = \{ a \in A \mid x \in a \}$,
  it follows that every $c \in \fun{V}\topo{X}$ is of the form ${\uparrow}x$
  for some $x \in \topo{X}$.
  As a consequence, the meet of $c$ and $c'$ in $\fun{V}\topo{X}$
  is
  $$
    c \wwedge c' = \{ y \wedge z \mid y \in c, z \in c' \}.
  $$
  To see this, write $c = {\uparrow}x_c$ and $c' = {\uparrow}x_{c'}$.
  Then $c \wwedge c'$ in $\fun{V}\topo{X}$ corresponds to
  $\tilde{x}_c \cap \tilde{x}_{c'} = \widetilde{x_c \wedge x_{c'}}$
  in $\TFil(\CFil\topo{X})$, so $c \wwedge c' = {\uparrow}(x_c \wedge x_{c'})$.
    
\begin{lem}\label{lem:V-I}
  The functors $\fun{V}$ and $\TFil \cdot \CFil$ are naturally isomorphic.
\end{lem}
\begin{proof}
  Define the transformation
  $(f_{\topo{X}})_{\topo{X} \in \cat{ISpace}} : \fun{V} \to \TFil \cdot \CFil$
  on objects as in Lemma \ref{lem:V-iso}. We show that it is natural.
  Since by Lemma \ref{lem:V-iso} it is isomorphic on objects, this proves
  the claimed natural isomorphism.

  Suppose $h : \topo{X} \to \topo{X}'$ is an I-space morphism.
  We need to show that
  $$
    \begin{tikzcd}
      \fun{V}\topo{X} 
            \arrow[r, "f_{\topo{X}}"]
            \arrow[d, "\fun{V}h" swap]
        & \TFil(\CFil\topo{X})
            \arrow[d, "\TFil(\CFil h)"]\\
      \fun{V}\topo{X} 
            \arrow[r, "f_{\topo{X}'}"]
        & \TFil(\CFil\topo{X})
    \end{tikzcd}
  $$
  commutes. To this end, let $c \in \fun{V}\topo{X}$ and
  $a' \in \CFil\topo{X}'$, and compute
  \begin{align*}
    a' \in f_{\topo{X}'}(\fun{V}h(c))
      \iff h[c] \subseteq a' 
      &\iff c \subseteq h^{-1}(a') \\
      &\iff h^{-1}(a') \in f_{\topo{X}}(c) 
      \iff a' \in \TFil(\CFil h)(f_{\topo{X}}(c)).
  \end{align*}
  This proves the proposition.
\end{proof}

  We now prove that descriptive $\Box$-frames are indeed $(\fun{I}, \fun{V})$-dialgebras.

\begin{thm}\label{thm:ISpace-Dialg}
  We have $\cat{ISpace_{\Box}} \cong \cat{Dialg}(\fun{I}, \fun{V})$.
\end{thm}
\begin{proof}
  We first establish the isomorphism on objects.
  
\begin{clm}
  For a descriptive $\Box$-frame $(X, \leq, R, A)$, let
  $\topo{X}$ be the I-space isomorphic to the descriptive I-frame $(X, \leq, A)$
  and define
  $$
    \gamma_R : \fun{I}\topo{X} \to \fun{V}\topo{X} : x \mapsto R[x].
  $$
  Then $(\topo{X}, \gamma_R)$ is a $(\fun{I}, \fun{V})$-dialgebra.
\end{clm}
\begin{subproof}[Proof of claim]
  We need to show that $\gamma_R$ is an M-space morphism from $\topo{X}$ to
  $\fun{V}\topo{X}$. It is well defined since the fact that $(X, \leq, R, A)$ is
  descriptive implies that $\gamma_R(x) = R[x]$ is a closed filter for each $x$.
  Besides, $\gamma_R$ preserves the top element because $\gamma_R(\top) = R[\top] = \{ \top \}$
  and the latter is the top element in $\fun{V}\topo{X}$.
  To see that it preserves binary meets we can compute
  $$
  \gamma_R(x) \wwedge \gamma_R(y)
    = R[x] \wwedge R[y]
    = \{ z \wedge w \mid z \in R[x], w \in R[y] \}
    = R[x \wedge y]
    = \gamma_R(x \wedge y)
  $$
  Here the third equality follows from \ref{it:box-frm3} and \ref{it:box-frm4}
  in Definition \ref{def:box-frm}.
  Finally, $\gamma_R$ is continuous because the topology on $\fun{V}\topo{X}$ is
  generated by subsets of the form $\dbox a$, where $a \in A$, and their complements,
  and $\gamma_R^{-1}(\dbox a) = m_{\Box}(a) \in A$ for all $a \in A$.
\end{subproof}

\begin{clm}
  Let $(\topo{X}, \gamma)$ be a $(\fun{I}, \fun{V})$-dialgebra.
  Let $(X, \leq, A)$ be the descriptive I-frame corresponding to $\topo{X}$ and
  define the relation $R_{\gamma}$ on $X$ by $xR_{\gamma}y$ iff $y \in \gamma(x)$.
  Then $(X, \leq, R, A)$ is a descriptive $\Box$-frame.
\end{clm}
\begin{subproof}[Proof of claim]
  This is indeed a general $\Box$-frame because for every clopen filter $a \in A$
  we have $m_{\Box}(a) = \gamma^{-1}(\dbox a) \in \CFil\topo{X} = A$.
  It is descriptive because $R_{\gamma}[x] = \gamma(x)$ is the intersection of
  clopen filters.
\end{subproof}

  It is easy to see that the two assignments define an isomorphism on objects:
  We already know the correspondence for the underlying I-spaces (descriptive I-frames),
  and on top of this we have $xRy$ iff $y \in \gamma_R(x)$ iff $xR_{\gamma_R}y$,
  and $y \in \gamma(x)$ iff $xR_{\gamma}y$ iff $y \in \gamma_{R_{\gamma}}x$.

  The verification of the isomorphism on morphisms is routine.
\end{proof}

  The proof of the following theorem was sketched at the start of this section.

\begin{thm}
  We have a dual equivalence of categories
  $$
    \cat{ISpace_{\Box}} \equiv^{\op} \cat{ISL_{\Box}}.
  $$
\end{thm}
\begin{proof}
  According to Theorem~\ref{thm:ISpace-Dialg} we have
  $\cat{ISpace_{\Box}} \cong \cat{Dialg}(\fun{I}, \fun{V})$.
  It follows from Lemma~\ref{lem:V-I} that $\fun{V}$ is naturally isomorphic
  to $\fun{I}$, so
  $\cat{Dialg}(\fun{I}, \fun{V})$ is equivalent to $\cat{Dialg}(\fun{I}, \fun{I})$.
  The dual nature of dialgebras, together with the facts that $\fun{I}$ and
  $\fun{j}$ are inclusion functors, implies that
  $\cat{Dialg}(\fun{I}, \fun{I})$ is dual to $\cat{Dialg}(\fun{j}, \fun{j})$.
  Since $\fun{j}$ is naturally isomorphic to $\fun{N}$ the latter is
  equivalent to $\cat{Dialg}(\fun{N}, \fun{j})$,
  which in turn is isomorphic to $\cat{ISL_{\Box}}$ by
  Proposition~\ref{prop:ISLb-dialg}.
\end{proof}

  As a corollary we obtain completeness for $\mathbf{MI_{\Box}}$. 

\begin{thm}
  The system $\mathbf{MI_{\Box}}$ is complete with respect to $\Box$-models
  and descriptive $\Box$-models.
\end{thm}
\begin{proof}
  The proof of this theorem proceeds in a similar way as for $\MI$
  in Section \ref{sec:base}.
\end{proof}

  The astute reader will have noticed that all functors used in the dialgebraic
  perspective explained above are (naturally isomorphic to) inclusion functors.
  We will encounter functors for which this is not the case in
  Section~\ref{sec:monotone}.

\section{Conservative extensions}\label{sec:fragments}

  There are many variations of modal intuitionistic logic,
  and usually the $\Box$-modality satisfies the axioms \ref{ax:box1} and \ref{ax:box2}
  (defined in Section \ref{sec:normal} above).
  Therefore, the various semantics of these versions of modal intuitionistic
  logic all provide sound semantics for $\mathbf{MI_{\Box}}$.
  In this section we show that in fact $\mathbf{MI_{\Box}}$ is complete with respect
  to many of these semantics. This then implies that $\mathbf{MI_{\Box}}$
  axiomatises the $(\top, \wedge, \to, \Box)$-fragment of the modal intuitionistic
  logics under consideration.
  
  Our strategy for achieving this is simple: we show that $\Box$-frames
  and -models are special cases of the semantics corresponding to various
  variations of modal intuitionistic logic. Then, letting $\cat{K}$ denote
  such a class of frames for modal intuitionistic logic, we argue that
  $\cat{K} \Vdash \phi$ implies $\cat{IF_{\Box}} \Vdash \phi$,
  which we know implies $\mathbf{MI_{\Box}} \vdash \phi$.

  We make frequent use of the following proposition.
  Here, if $R_1$ and $R_2$ are two relations on a set $X$, then we write
  $R_1 \circ R_2 := \{ (x, z) \in X \times X \mid \exists y \in X \text{ s.t. } xR_1y \text{ and } y R_2 z \}$.

\begin{prop}\label{prop:rels}
  Let $(X, \leq, R)$ be a $\Box$-frame. Then we have
  $$
    R = ({\leq} \circ R \circ {\leq})
    \quad\text{ and }\quad
    ({\leq}^{-1} \circ R) \subseteq (R \circ {\leq}^{-1}).
  $$
\end{prop}
\begin{proof}
  We first prove the left equality.
  The inclusion from left to right follows from reflexivity of $\leq$. Furthermore,
  it follows from \ref{it:box-frm2} that $R = (R \circ {\leq})$,
  so it suffices to show that $({\leq} \circ R) \subseteq R$. In a diagram:
  $$
    \begin{tikzcd}[arrows=-]
      y \arrow[d, "\leq" swap] \arrow[r, "R"]
        & z \\
      x \arrow[ru, dashed, "R" swap]
    \end{tikzcd}
  $$
  If $x \leq y$ then $x \wedge y = x$. Also, by \ref{it:box-frm1}
  we have $xR\top$. It then follows from \ref{it:box-frm3} that
  $(x \wedge y)R(z \wedge \top)$, i.e.~$xRz$.
  
  The inclusion on the right follows from \ref{it:box-frm4}: if $xRz$ and $x \leq x'$
  then $(x \wedge x')Rz$ so according to \ref{it:box-frm4} we can find a states $y, y'$
  such that $xRy$ and $x'Ry'$ and $y \wedge y' = z$. The latter implies
  $z \leq y'$, so $y'$ witnesses the desired inclusion.
\end{proof}

  We now consider various flavours of modal intuitionistic logic.
  We write $\Int_{\Box}$ and
  $\Int_{\Box\Diamond}$ for the extensions of
  intuitionistic logic with $\Box$, and with $\Box$ and $\Diamond$,
  respectively.
 
\subsection*{The systems \texorpdfstring{$\mathbf{H\Box}$}{HBox} and \texorpdfstring{$\mathbf{H{\Box\Diamond}}$}{HBoxDiamond}
           by Bo\v{z}i\'{c} and Do\v{s}en}
  In \cite{BozDos84} the authors consider the logic $\mathbf{HK\Box}$,
  which is an extension of intuitionistic logic with $\Box$ axiomatised by
  \ref{ax:box1} and \ref{ax:box2}.
  
  The interpreting structures are \emph{$H\Box$-frames} \cite[Definition 2]{BozDos84}
  (found under a different name in \cite[Definition 3.1]{Pro12}).
  A $H\Box$-frame is a tuple $(X, \leq, R)$ consisting of a pre-order $(X, \leq)$
  and a relation $R$ on $X$ satisfying $({\leq} \circ R) \subseteq (R \circ {\leq})$.
  Such a frame is called \emph{condensed} if $(R \circ {\leq}) = R$ and
  \emph{strictly condensed} if $({\leq} \circ R \circ {\leq}) = R$.
  They become \emph{$H\Box$-models} if one attaches to it a valuation
  that interprets each proposition letter as an upset in $(X, \leq)$.
  The box-modality is interpreted as usual, i.e.~a state $x$ satisfies
  $\Box\phi$ if all its $R$-successors satisfy~$\phi$.
  
  It follows from Proposition \ref{prop:rels} and the fact that a valuation for
  a $\Box$-frame sends a proposition letter to a filter, which in particular is
  an upset, that every $\Box$-model is also a strictly condensed $H\Box$-model.
  Therefore $\mathbf{MI_{\Box}}$ is sound and complete with respect to
  the class of ((strictly) condensed) $H\Box$-frames and
  $\mathbf{MI_{\Box}}$ characterises the $(\top, \wedge, \to, \Box)$-fragment
  of $\mathbf{HK\Box}$. As a consequence we have:

\begin{thm}
  The system $\mathbf{HK}\Box$ is a conservative extension of $\mathbf{MI_{\Box}}$.
\end{thm}

  Furthermore, in \S11 the authors also introduce the system $\mathbf{HK{\Box\Diamond}}$
  for the language $\Int_{\Box\Diamond}$.
  A semantics is given by the collection of so-called $H{\Box}{\Diamond}$-frames (models).
  These are strictly condensed $H\Box$-frames (models) that additionally satisfy
  $$
    ({\geq} \circ R) \subseteq (R \circ {\geq}).
  $$
  Again, by Proposition \ref{prop:rels}, every $\Box$-model is a
  $H{\Box\Diamond}$-model and $\mathbf{MI_{\Box}}$ axiomatises the
  $(\top, \wedge, \to, \Box)$-fragment of $\mathbf{HK{\Box\Diamond}}$.
  Therefore:

\begin{thm}
  The system $\mathbf{HK}\Box\Diamond$ is a conservative extension of $\mathbf{MI_{\Box}}$.
\end{thm}

\subsection*{The system \texorpdfstring{$\mathbf{IK}$}{IK} by Plotkin and Stirling}
  In \cite{PloSti86} the authors define an \emph{intuitionistic modal frame}
  as a tuple $(X, \leq, R)$ consisting of a poset $(X, \leq)$ and a relation $R$
  satisfying
  $$
    ({\geq} \circ R) \subseteq (R \circ {\geq})
    \quad\text{ and }\quad
    (R \circ {\leq}) \subseteq ({\leq} \circ R).
  $$
  These can be turned into models in the usual way.
  While the diamond is interpreted as usual (i.e.~a state satisfies
  $\Diamond\phi$ if is has an $R$-successor satisfying $\phi$),
  persistence of $\Box$ is achieved by defining
  $$
    x \Vdash \Box\phi \iff x({\leq} \circ R)y \text{ implies } y \Vdash \phi.
  $$
  When a frame satisfies $({\leq} \circ R) \subseteq R$ then reflexivity of
  $\leq$ implies that $({\leq} \circ R) = R$.
  In such cases, the interpretation of $\Box$ reduces to the usual one, i.e.
  $$
    x \Vdash \Box\phi \iff xRy \text{ implies } y \Vdash \phi.
  $$
  This interpretation coincides with the one given by Fisher Servi in \cite{Fis80}.

  It follows from Proposition \ref{prop:rels} again that $\Box$-frames
  are intuitionistic modal frames. Moreover, the fact that every $\Box$-frame
  satisfies $R = ({\leq} \circ R \circ {\leq})$ implies $({\leq} \circ R) \subseteq R$
  so in case of $\Box$-frames our interpretation coincides with the one
  used in \cite{PloSti86}.
  Therefore we have completeness of $\mathbf{MI_{\Box}}$
  with respect to intuitionistic modal frames.
  Again, soundness follows from the axiomatisation of the system
  $\mathbf{IK}$ presented in \cite{PloSti86}.
  So $\mathbf{MI_{\Box}}$ axiomatises the
  $(\top, \wedge, \to, \Box)$-fragment of $\mathbf{IK}$.

\begin{thm}
  The system $\mathbf{IK}$ is a conservative extension of $\mathbf{MI_{\Box}}$.
\end{thm}

\subsection*{The system \texorpdfstring{$\mathbf{IntK_{\Box}}$}{IntKBox} by Wolter and Zakharyaschev}
  The frames used in \cite[Section 2]{WolZak98} to interpret
  $\Int_{\Box}$
  are posets $(X, \leq)$ with an additional relation $R$ satisfying
  $R = ({\leq} \circ R \circ {\leq})$.
  (The same frames have also been used in~\cite{Yok85} and~\cite{Koj11}.)
  Formulae are interpreted as upsets
  in $(X, \leq)$ and models are obtained by adding to a frame a valuation
  that sends each proposition letter to an up-closed subset of the frame.
  The axiomatisation of $\mathbf{IntK_{\Box}}$ proves
  the axioms $\Box\phi \wedge \Box\psi = \Box(\phi \wedge \psi)$ and $\Box\top = \top$.
  It can be proved similarly as above that $\mathbf{MI_{\Box}}$ is sound
  and complete with respect to their models, and that
  $\mathbf{MI_{\Box}}$ axiomatises the $(\top, \wedge, \to, \Box)$-fragment
  of $\mathbf{IntK_{\Box}}$.

\begin{thm}
  The system $\mathbf{IntK_{\Box}}$ is a conservative extension of $\mathbf{MI_{\Box}}$.
\end{thm}

\section{Dialgebraic perspective}\label{sec:dialg}

  We have seen in Sections \ref{sec:normal} and \ref{sec:duality} that the
  categories of $\Box$-frames, descriptive $\Box$-frames and ISLOs are all
  categories of dialgebras.
  In this section we will see that in fact $\mathbf{MI_{\Box}}$ is an instance
  of a \emph{dialgebraic logic} \cite{GroPat20}.
  Using general results from {\it op.~\!cit.}~we 
  then re-obtain the completeness result for $\mathbf{MI_{\Box}}$.
  Furthermore, the dialgebraic view allows us to transfer the general
  expressivity-somewhere-else result for dialgebraic logics from
  \cite[Section 7]{GroPat20} to the setting of modal meet-implication logic:
  Two states $x$ and $x'$ in a $\Box$-frame are logically equivalent if and only if
  corresponding states $\eta(x)$ and $\eta(x')$ in the so-called
  \emph{filter extension} are behaviourally equivalent.
  
  In Section \ref{sec:monotone} we use the dialgebraic perspective 
  to succinctly present the extension of meet-implication logic with
  a monotone unary modality.
  
  We begin by recalling the framework of dialgebraic logic, specialised
  to our setup.
  In an effort to avoid notational clutter, we work in a
  setting without proposition letters.
  
\subsection*{Dialgebraic logic}
  We sketch the basic setup of dialgebraic logic \cite[Section 5]{GroPat20}
  specialised to our base logic. This starts with 
  a dual adjunction and a restriction thereof. In our case:
  $$
    \begin{tikzcd}
      \cat{MF}
            \arrow[r, {Rays[]}->, shift left=2.5pt, "\fun{F}"]
        & \cat{SL}
            \arrow[l, {Rays[]}->, shift left=2.5pt, "\fun{F}"] \\
      \cat{IF}
            \arrow[u, "\fun{i}" left]
            \arrow[r, shift left=2.5pt, "\fun{F}_1"]
        & \cat{ISL}
            \arrow[u, "\fun{j}" right]
            \arrow[l, shift left=2.5pt, "\fun{F}_2"]
    \end{tikzcd}
  $$
  Recall that $\cat{MF}$ is simply the same category of $\cat{SL}$,
  but viewed as frame semantics.
  Both units of the dual adjunction in the top row are given by
  $\eta_X : X \to \fun{F}X : x \mapsto \tilde{x} = \{ a \in \fun{F}X \mid x \in a \}$.
  The functors $\fun{F}_1$ and $\fun{F}_2$ are restrictions of $\fun{F}$
  to $\cat{IF}$ and $\cat{ISL}$ respectively.
  It follows from Corollary \ref{cor:imp-imp} that $\fun{F}_1$ and $\fun{F}_2$
  are well defined on objects.
  An easy computation shows that applying $\fun{F}$ to an I-frame morphism
  yields an $\cat{ISL}$-homomorphism, so $\fun{F}_1$ is well defined.
  For the converse, we note that $\fun{F}_2 = \fun{U}_{\cat{IF}} \circ \TFil$,
  where $\fun{U}_{\cat{IF}} : \cat{ISpace} \to \cat{IF}$ is the obvious forgetful
  functor. This implies that $\fun{F}_2$ is well defined.
  
  Furthermore, the latter proves that for every implicative semilattice $A$,
  the unit $\eta_A : A \to \fun{F}_1\fun{F}_2A$ is an implicative semilattice
  homomorphism.
  In the terminology of \cite{GroPat20} this means that the setup is
  \emph{well structured}, and this is a prerequisite for many of the results
  in {\it op.~\!cit}.

  Let $\fun{T} : \cat{IF} \to \cat{MF}$ be a functor.
  A \emph{modal logic} for $\cat{Dialg}(\fun{i}, \fun{T})$ is a pair
  $(\fun{L}, \rho)$ where $\fun{L}$ is a functor $\cat{ISL} \to \cat{SL}$
  and $\rho$ is a natural transformation $\rho : \fun{LF}_1 \to \fun{FT}$.
  In a diagram:
  $$
    \begin{tikzcd}[row sep=large, column sep=large]
      \cat{MF}
            \arrow[r, {Rays[]}->, shift left=2.5pt, "\fun{F}"]
        & \cat{SL}
            \arrow[l, {Rays[]}->, shift left=2.5pt, "\fun{F}"] \\
      \cat{IF}
            \arrow[u, "\fun{T}" left]
            \arrow[r, {Rays[]}->, shift left=2.5pt, "\fun{F}_1"]
        & \cat{ISL}
            \arrow[u, "\fun{L}" right]
            \arrow[l, {Rays[]}->, shift left=2.5pt, "\fun{F}_2"]
            \arrow[lu, Rightarrow, shorten <=1em, shorten >=1em, "\rho" {swap,pos=.4}]
    \end{tikzcd}
  $$
  The natural transformation $\rho$ gives rise to the abstract notion
  of a ``complex algebra'', which assigns to a $(\fun{i}, \fun{T})$-dialgebra $(X, \gamma)$
  the $(\fun{L}, \fun{j})$-dialgebra $(\fun{F}_1X, \gamma^*)$ given by the composition
  $$
    \begin{tikzcd}
      \fun{LF}_1X
            \arrow[r, "\rho_X"]
        & \fun{FT}X
            \arrow[r, "\fun{F}\gamma"]
        & \fun{Fi}X
            \arrow[r, "\cong"]
        & \fun{jF}_1X.
    \end{tikzcd}
  $$
  We assume the existence of an initial object $(\Psi, \ell)$ in $\cat{Dialg}(\fun{N}, \fun{j})$,
  which plays the r\^{o}le of Lindenbaum-Tarski algebra.
  (We construct an explicit example of the initial dialgebra in case $\fun{L} = \fun{N}$ below.)
  The interpretation
  of formulae in an $(\fun{i}, \fun{T})$-dialgebra $(X, \gamma)$ is then given
  by the unique morphism $\llb \cdot \rrb_{\gamma} : \Psi \to X$ in $\cat{ISL}$
  making the following diagram commute:
  $$
    \begin{tikzcd}[arrows=-latex]
      \fun{L}\Psi
            \arrow[r, "\fun{L}\llb \cdot \rrb_{\gamma}"]
            \arrow[d, "\ell" left]
        & \fun{LF}_1X
            \arrow[d, "\gamma^*"] \\
      \fun{j}\Psi
            \arrow[r, "\fun{j}\llb \cdot \rrb_{\gamma}"]
        & \fun{j}\fun{F}_1X
    \end{tikzcd}
  $$
  
  Specialising to $\mathbf{MI_{\Box}}$, we take $\fun{T} = \fun{G}$
  (cf.~Definition \ref{def:fun-G}) and $\fun{L} = \fun{N}$ (Section \ref{sec:normal}),
  and define $\rho : \fun{NF}_1 \to \fun{FG}$ on components by
  $$
    \rho_{(X, \leq)}(\Box a) \mapsto \{ p \in \fun{G}(X, \leq) \mid p \subseteq a \}.
  $$
  The initial dialgebra in $\cat{Dialg}(\fun{N}, \fun{j})$ is the
  Lindenbaum-Tarski algebra corresponding to $\mathbf{MI_{\Box}}$.
  Explicitly, write $\Psi$ for this Lindenbaum-Tarski algebra
  (i.e.~$\ms{L}_{\Box}$-formulae modulo provable equivalence) and let
  $[\phi]$ denote the equivalence class of a formula $\phi$.
  Then we can define an $(\fun{N}, \fun{j})$-dialgebra structure on $\Psi$
  by setting $\ell(\Box[\phi]) = [\Box\phi]$. For any $(\fun{N}, \fun{j})$-dialgebra
  $(A, \alpha)$ the map that sends $[\top] \in \Psi$ to $\top_A$, the top element in $A$,
  uniquely extends to an $(\fun{N}, \fun{j})$-dialgebra morphism
  $i : (\Psi, \ell) \to (A, \alpha)$ via
  \begin{align*}
    i([\phi] \wedge [\psi]) &= i([\phi]) \wedge_A i([\psi]), \\
    i([\phi] \to [\psi]) &= i([\phi]) \mto_A i([\psi]), \\
    i([\Box\phi]) &= \alpha(\Box_A(i([\phi]))).
  \end{align*}
  In case $(A, \alpha)$ is the complex dialgebra of a $(\fun{i}, \fun{G})$-dialgebra
  $(X, \gamma)$, the meet in $A$ is given by intersection, and implication by
  $$
    i([\phi]) \mto_A i([\psi])
      = \{ x \in X \mid \text{if } x \leq y
                      \text{ and } y \in i([\phi])
                     \text{ then } y \in i([\psi]) \}.
  $$
  
  By construction, for any $(\fun{i}, \fun{G})$-dialgebra $(X, \leq, \gamma)$
  the following commutes:
  $$
    \begin{tikzcd}
      \fun{NF}_1(X, \leq)
            \arrow[r, "\rho_X"]
        & \fun{F}\fun{G}(X, \leq)
            \arrow[r, "\fun{F}_2\gamma"]
        & \fun{F}\fun{i}(X, \leq)
            \arrow[r, phantom, "\cong"]
        & [-2em]
          \fun{jF}_1(X, \leq) \\
      \fun{N}\Psi
            \arrow[u, "\fun{N}\llb \cdot \rrb_{\gamma}"]
            \arrow[rrr, "\ell"]
        &
        &
        & \fun{j}\Psi
            \arrow[u, "\fun{j}\llb \cdot \rrb_{\gamma}"]
    \end{tikzcd}
  $$
  Let us compute explicitly what this means for the interpretation of
  a formula of the form $\Box\phi$.
  The interpretation $\llb \Box\phi \rrb_{\gamma}$ of $\Box\phi$ in $(X, \gamma)$
  is then given by $\fun{j}\llb \ell(\Box\phi) \rrb_{\gamma}$.
  By commutativity of the diagram this equals
  \begin{align*}
    \fun{F}_2\gamma(\rho_{(X,\leq)}(\fun{N}\llb \cdot \rrb_{\gamma}(\Box\phi)))
      &= \gamma^{-1}(\rho_X(\Box \llb \phi \rrb_{\gamma})) \\
      &= \gamma^{-1}(\{ b \in \fun{F}(X, \leq) \mid b \subseteq \llb \phi \rrb_{\gamma} \}) \\
      &= \{ x \in X \mid \gamma(x) \subseteq \llb \phi \rrb_{\gamma} \}
  \end{align*}
  Indeed, a state $x$ satisfies $\Box\phi$ iff all of its successors (all $y \in \gamma(x)$)
  satisfy $\phi$, as desired.
  
\subsection*{Logic via predicate liftings and axioms}
  A convenient way to think about dialgebraic logics is via predicate liftings
  and axioms.
  Given a collection of so-called predicate liftings and axioms
  between them, we can define a logic $(\fun{L}, \rho)$.

  A (unary) \emph{predicate lifting} is a natural transformation 
  $$
    \lambda : \fun{UFi} \to \fun{UFT},
  $$
  where $\fun{U} : \cat{SL} \to \cat{Set}$ denotes the forgetful functor.
  For a collection $\Lambda$ of predicate liftings for $\fun{T}$ and a set $X$,
  let $\fun{L}'X$ be the free semilattice generated by the set
  $$
    \big\{ \heartsuit^{\lambda}v \mid \lambda \in \Lambda, v \in \Var \big\}.
  $$
  By a \emph{$\Lambda$-axiom} we mean a pair
  $(\phi, \psi)$ of two elements in $\fun{L}'\Var$, where $\Var$ is some set of variables.

  Let $\Lambda$ be a set of predicate liftings and $\Ax$ a collection of
  $\Lambda$-axioms.
  The functor $\fun{L}$ then arises by sending an implicative meet-semilattice $A$
  to the free semilattice generated by $\{ \heartsuit^{\lambda}(a) \mid a \in A \}$
  modulo the relations $R$ from $\Ax$, where the variables are substituted for elements of $A$.
  For an $\cat{ISL}$-homomorphism $h : A \to B$, the morphism
  $\fun{L}f$ is defined on generators by
  $\fun{L}f(\heartsuit^{\lambda}a) = \heartsuit^{\lambda}f(a)$.
  The induced natural transformation $\rho$ is given on components by
  $\rho_{(X, \leq)}(\heartsuit^{\lambda}a) = \lambda_{(X, \leq)}(a)$.

  A collection $\Ax$ of axioms is said to be \emph{sound} if it induces
  a well-defined natural transformation
  $\rho_{(\Lambda, \Ax)} : \fun{LP'} \to \fun{PT}$ via
  $$
    \rho_{(\Lambda, \Ax),X}([\heartsuit^{\lambda}(a_1, \ldots, a_n)]_R)
      = \lambda_X(a_1, \ldots, a_n)
  $$
  that is independent of the choice of representative relative to
  $R$-equivalence classes $[ \cdot ]_R$.
  If this is well defined, naturality is
  an immediate consequence of naturality of the predicate liftings.
  
  In the case of $\mathbf{MI_{\Box}}$, 
  we can view the modality $\Box$ as being induced by the predicate lifting
  $\lambda : \fun{UFi} \to \fun{UFG}$ given by
  $$
    \lambda_{(X, \leq)} : \fun{UFi}(X, \leq) \to \fun{UFG}(X, \leq)
                        : p \mapsto \{ q \in \fun{G}A \mid q \subseteq p \}.
  $$
  Let us temporarily write $\heartsuit^{\lambda}$ for the modal operator
  induced by $\lambda$.
  For a $\Box$-frame $\mo{M} = (X, \leq, \gamma)$ we then have
  $$
    \mo{M}, x \Vdash \heartsuit^{\lambda}\phi
      \iff \gamma(x) \in \lambda_{(X, \leq)}(\llb \phi \rrb)
      \iff \gamma(x) \subseteq \llb \phi \rrb.
  $$
  Recognising $\gamma(x)$ as the successor-set of $x$, this shows that
  $x$ satisfies $\heartsuit^{\lambda}\phi$ if all its successors satisfy $\phi$.
  So $\heartsuit^{\lambda}$ coincides with $\Box$ as defined above.
  
  We verify that $\lambda$ is indeed a well-defined natural transformation.

\begin{prop}
  $\lambda : \fun{UFi} \to \fun{UFG}$ is a well-defined natural transformation.
\end{prop}
\begin{proof}
  Let $(X, \leq)$ be an I-frame and $p \in \fun{UFi}(X, \leq)$.
  Since $\fun{G}(X, \leq)$ is ordered by reverse inclusion
  the set $\lambda_{(X, \leq)}(p)$ is up-closed in $\fun{G}(X, \leq)$.
  Moreover, if $q, r \in \lambda_{(X, \leq)}(p)$, then $q \cup r \subseteq p$.
  This implies $q \wedge_{\fun{G}} r \subseteq p$,
  so $\lambda_{(X, \leq)}(p)$ is closed under meets in $\fun{G}(X, \leq)$,
  and hence $\lambda_{(X, \leq)}(p)$ is a filter in $\fun{G}(X, \leq)$.
  Next we show that $\lambda_{(X, \leq)}$ is a homomorphism in $\cat{SL}$.
  Clearly $\lambda_{(X, \leq)}(X) = \fun{G}(X, \leq)$, so top is preserved.
  Let $p, q$ be filters in $(X, \leq)$, then
  $$
    \lambda_{(X, \leq)}(p \cap q)
      = \{ r \mid r \subseteq p \cap q \}
      = \{ r \mid r \subseteq p \} \cap \{ r \mid r \subseteq q \}
      = \lambda_{(X, \leq)}(p) \cap \lambda_{(X, \leq)}(q).
  $$
  Finally, for naturality, let $h : (X, \leq) \to (X', \leq')$ be an
  I-frame morphism and let $p'$ be a filter in $(X', \leq')$ and
  $q \in \fun{G}(X, \leq)$. Then we have
  \begin{align*}
    q \in \lambda_{(X, \leq)}(h^{-1}(p'))
      &\iff q \subseteq h^{-1}(p') \\
      &\iff h[q] \subseteq p' \\
      &\iff h[q] \in \lambda_{(X', \leq')}(p') \\
      &\iff q \in (\fun{G}h)^{-1}(\lambda_{(X', \leq')}(p'))
  \end{align*}
  so indeed $\lambda$ is a natural transformation.
\end{proof}

  Together with the axioms \ref{ax:box1} and \ref{ax:box2}
  we recover the functor $\fun{N} : \cat{ISL} \to \cat{SL}$.

\subsection*{Instatiating the theory of Dialgebraic Logic}
  Now that we have set up the logic $\mathbf{MI_{\Box}}$ in a dialgebraic setting,
  we can exploit this by using theorems from \cite{GroPat20}
  to obtain results for $\mathbf{MI_{\Box}}$ interpreted in $\Box$-frames.
  
  We start by re-obtaining completeness.
  By Propositions 6.9 and 6.10 in {\it op.~\!cit.}~it suffices
  to prove that the adjoint buddy $\rho^{\flat} : \fun{GF}_2 \to \fun{FN}$
  of $\rho$ has a right-inverse on components. This
  ``adjoint buddy'' is defined by the composition
  $$
    \begin{tikzcd}
      \fun{GF}_2
            \arrow[r, "\eta_{\fun{GF}_2}"]
        & \fun{FFGF}_2
            \arrow[r, "\fun{F}\rho_{\fun{F}_2}"]
        & \fun{FNF}_1\fun{F}_2
            \arrow[r, "\fun{FN}\eta'"]
        & \fun{FN}
    \end{tikzcd}
  $$
  and exists because our setup is well-structured.
  We shall establish a stronger statement: we can find
  a \emph{natural} transformation $\tau : \fun{FN} \to \fun{GF}_2$ such that
  $\rho^{\flat} \circ \tau = \id$.
  Guided by the duality for for $(\fun{N}, \fun{j})$-dialgebras from
  Section~\ref{sec:duality}, for $A \in \cat{ISL}$ define
  $$
    \tau_A : \fun{FN}A \to \fun{GF}_2A
           : U \mapsto \{ p \in \fun{F}_2A \mid \text{if } \Box a \in U \text{ then } a \in p \}.
  $$
  This is easily seen to send $U$ to 
  the upwards closure of the filter $p_U = \{ a \in A \mid {\Box} a \in U \} \in \fun{F}_2A$.
  Therefore $\tau_A$ is well defined, i.e.~it lands in $\fun{GF}_2A$.
  It clearly preserves the top element, and it preserves binary meets because
  $$
    \tau_A(U) \wwedge \tau_A(V)
      = {\uparrow}_{\subseteq}p_U \wwedge {\uparrow}_{\subseteq}p_V
      = {\uparrow}_{\subseteq}(p_U \cap p_V)
      = {\uparrow}_{\subseteq}(p_{U \cap V})
      = \tau_A(U \cap V).
  $$
  
\begin{prop}\label{prop:i-G-tau}
  The assignment $\tau = (\tau_A)_{A \in \cat{ISL}} : \fun{FN} \to \fun{GF}_2$
  defines a natural transformation and satisfies $\rho^{\flat} \circ \tau = \id_{\fun{FN}}$.
\end{prop}
\begin{proof}
  We first prove naturality. This means that we have to show that
  $$
    \begin{tikzcd}
      \fun{FN}A
            \arrow[r, "\tau_A"]
        & \fun{GF}_2A \\
      \fun{FN}B
            \arrow[r, "\tau_B"]
            \arrow[u, "\fun{FN}h" left]
        & \fun{GF}_2B
            \arrow[u, "\fun{GF}_2h" right]
    \end{tikzcd}
  $$
  commutes for any implicative meet-semilattice homomorphism $h : A \to B$.
  Suppose $U \in \fun{FN}B$ and $p \in \tau_A(\fun{FN}h(U))$.
  In order to show that $p \in \fun{GF}_2h(\tau_B(U))$ we need to find
  $q \in \tau_B(U)$ such that $\fun{F}_2h(q) = p$. Since $\fun{F}_2h$ is
  a bounded morphism, it suffices to prove that there exists
  $q \in \tau_B(U)$ such that $\fun{F}_2h(q) \subseteq p$.
  Consider $q_U = \{ b \in B \mid \Box b \in U \}$.
  By construction $q_U \in \tau_B(U)$.
  Furthermore, $a \in \fun{F}_2h(q_U)$ implies $h(a) \in q_U$,
  hence $\Box h(a) \in U$ so $a \in \fun{FN}h(U)$ and since
  $p \in \tau_A(\fun{FN}h(U))$ this entails that $a \in p$.
  So $\fun{F}_2h(q_U) \subseteq p$, as desired.
  
  Conversely, suppose $p \in \fun{GF}_2h(\tau_B(U))$.
  Then $p = \fun{F}_2h(q)$ for some $q \in \tau_B(U)$.
  Also, $a \in \fun{FN}h(U)$ implies $h(a) \in U$, hence $h(a) \in q$ which
  in turn implies $a \in p$. Therefore $p \in \tau_A(\fun{FN}h(U))$, as desired.
  This proves commutativity of the diagram, and hence naturality of $\tau$.
  
  In order to prove that $\rho^{\flat}_A \circ \tau_A = \id$ we first note
  that for any $U \in \fun{F}NA$ and $a \in A$ we have
  $\tau(U) \subseteq \tilde{a}$ if and only if $\Box a \in U$.
  The direction from right to left follows from the definition of $\tau$.
  For the converse direction suppose $\Box a \notin U$,
  then $p_U$ is such that $p_U \in \tau(U)$ while $p_U \notin \tilde{a}$,
  so that $\tau(U) \not\subseteq \tilde{a}$.
  
  Now to see that $\rho^{\flat}_A \circ \tau_A = \id$, note that elements of
  $\fun{FN}A$ are filters of $\fun{N}A$, and these are uniquely determined
  by the collection of elements of the form $\Box a$ they contain, where $a \in A$.
  We now have
  $$
    \Box a \in \rho^{\flat}(\tau(U))
      \iff \tau(U) \subseteq \tilde{a}
      \iff \Box a \in U.
  $$
  The first ``iff'' follows from unravelling the definition of $\rho^{\flat}$.
\end{proof}

  In \cite[Definition 6.1]{GroPat20}, a logic $(\fun{L}, \rho)$ for $(\fun{i}, \fun{T})$-dialgebras is called
  \emph{complete} if the source
  \begin{equation}\label{eq:source}
    \big\{ \llb \cdot \rrb_{\gamma} : \Psi \to \fun{F}_1(X, \leq) \big\}_{(X, \leq, \gamma) \in \cat{Dialg}(\fun{i}, \fun{T})}
  \end{equation}
  is jointly monic. Intuitively, this means that for every two non-equivalent
  formulae $\phi, \psi \in \Psi$ we can find a $(\fun{i}, \fun{T})$-dialgebra
  $(X, \leq, \gamma)$ in which $\llb \phi \rrb_{\gamma} \neq \llb \psi \rrb_{\gamma}$.
  In our setup this reduces to the usual notion of (weak) completeness.
  As a corollary of Propositions 6.9 and 6.10 in \cite{GroPat20}
  we get:
  
\begin{thm}[Completeness]
  The system $\mathbf{MI_{\Box}}$ is complete with respect to
  $(\fun{i}, \fun{G})$-dialgebras.
\end{thm}

  Concretely, $\tau$ allows us to transform a $(\fun{N}, \fun{j})$-dialgebra $(A, \alpha)$
  into a $(\fun{i}, \fun{G})$-dialgebra whose complex algebra has $(A, \alpha)$
  as a subalgebra. Unravelling the constructions
  shows that this is precisely the $(\fun{i}, \fun{G})$-dialgebra
  underlying the dual of descriptive $\Box$-frame of $(A, \alpha)$.

  Letting $(X, \leq, \gamma, \Psi)$ be the dual dual of $(\Psi, \ell)$
  then gives rise to a $\fun{G}$-coalgebra $(X, \leq, \gamma)$
  such that $(\Psi, \ell)$ is a subalgebra of the complex algebra
  $(\fun{F}_1(X, \leq), \gamma^*)$ of $(X, \leq, \gamma)$.
  In particular this implies the existence of a monomorphism $m : \Psi \to \fun{F}_1X$
  such that
  $$
    \begin{tikzcd}
      \fun{N}\Psi
            \arrow[d, "\ell" swap]
            \arrow[r, "\fun{N}m"]
        & \fun{NF}_1(X, \leq)
            \arrow[d, "\gamma^*"] \\
      \fun{j}\Psi
            \arrow[r, "\fun{j}m"]
        & \fun{jF}_1(X, \leq)
    \end{tikzcd}
  $$
  commutes.
  Initiality of $(\Psi, \ell)$ then implies that $m = \llb \cdot \rrb_{\gamma} : \Psi \to \fun{F}_1(X, \leq)$, and its monotonicity witnesses that (joint) monotonicity
  of the source in \eqref{eq:source}.

  Apart from completeness, the natural transformation $\tau$ entails results for
  expressivity and filter extensions. We call a $\Box$-frame \emph{expressive}
  if two states are logically equivalent if and only if they are behaviourally
  equivalent.
  If we translate this to the abstract dialgebraic setting, we get that an
  $(\fun{i}, \fun{T})$-dialgebra $(X, \leq, \gamma)$ is
  \emph{expressive} if the theory map $\oth_{\gamma}$ given by
  $$
    \begin{tikzcd}
      \fun{i}(X, \leq)
            \arrow[r, "\eta_{\fun{i}(X,\leq)}"]
        & \fun{FFi}(X, \leq) = \fun{iF}_2\fun{F}_1(X, \leq)
            \arrow[r, "\fun{iF}_2\llb \cdot \rrb_{\gamma}"]
        & [1em]
          \fun{iF}_2\Psi
    \end{tikzcd}
  $$
  factors through a $(\fun{i}, \fun{T})$-dialgebra morphism followed
  by a monomorphism in $\cat{IF}$.
  Intuitively, the theory map sends a state in a $(\fun{i}, \fun{T})$-dialgebra
  to the collection of formulae it satisfies.
  Monotonicity of the second part of the factorisation implies that
  logically equivalent states must already be identified by the dialgebra
  morphism.
  
  Such a strong property is not always attainable.
  A weaker condition is expressivity-somewhere-else:
  A $\Box$-frame $(X, \leq, \gamma)$ is said to be expressive-somewhere-else if there
  exists an expressive $\Box$-frame $(Y, \leq_Y, \delta)$ and a monotone
  truth-preserving map $f : X \to Y$.
  In that case, two states $x, x' \in X$ are logically equivalent if and only if
  $f(x)$ and $f(x')$ are behaviourally equivalent (in $(Y, \leq_Y, \delta)$).
  In general, a $(\fun{i}, \fun{T})$-dialgebra $(X, \leq, \gamma)$ is said to be
  \emph{expressive-somewhere-else} if the theory map $\oth_{\gamma}$
  factors through a morphism in $\cat{MF}$ followed by the theory
  map of an expressive $(\fun{i}, \fun{T})$-dialgebra.

  As a consequence of Proposition \ref{prop:i-G-tau} and
  \cite[Theorem 7.5]{GroPat20} we have:
  
\begin{thm}
  Every descriptive $\Box$-frame is expressive, and every
  $(\fun{i}, \fun{G})$-dialgebra is expressive-somewhere-else.
\end{thm}

  In classical modal logic, the r\^{o}le of $(Y, \leq_Y, \delta)$ is played
  by \emph{ultrafilter extensions},
  and in the context of intuitionistic logic by \emph{prime filter extensions}
  \cite{Gol05}, \cite[Section 7]{GroPat20}.
  In the setting of the current paper, the r\^{o}le of $(Y, \leq_Y, \delta)$ is
  fulfilled by the \emph{filter extension} of $(X, \leq, \gamma)$.
  Chasing the definitions from \cite{GroPat20},
  we find that the filter extension of a $(\fun{i}, \fun{G})$-dialgebra 
  $(X, \leq, \gamma)$ is the $(\fun{i}, \fun{G})$-dialgebra underlying the
  descriptive $\Box$-frame dual to the complex algebra $(\fun{F}_1(X, \leq), \gamma^*)$.
  Concretely, the filter extension of $(X, \leq, \gamma)$ has
  state-space $(\hat{X}, \subseteq) = \fun{F}_2\fun{F}_1(X, \leq)$ and its
  structure map given by
  $$
    \hat{\gamma}
      : \fun{i}(\hat{X}, \subseteq) \to \fun{G}(\hat{X}, \subseteq)
      : p \mapsto \bigcap \{ \tilde{a} \subseteq \hat{X} \mid \gamma^{-1}(\dbox a) \in p \}.
  $$
  Here $\dbox a = \{ b \in \fun{F}(X, \leq) \mid b \subseteq a \}$,
  so that $\gamma^{-1}(\dbox a) \in \fun{F}(X, \leq)$.
  
  As a consequence of the results in \cite[Section 7]{GroPat20} the
  $(\fun{i}, \fun{G})$-dialgebra $(\hat{X}, \subseteq, \hat{\gamma})$ is
  expressive, and the theory map $\oth_{\gamma}$ of $(X, \gamma)$ of factors via
  $$
    \begin{tikzcd}
      \fun{i}(X, \leq)
            \arrow[rr, bend left=20, "\oth_{\gamma}"]
            \arrow[r, "\eta_{\fun{i}(X, \leq)}" below]
        & [1em]
          \fun{i}(\hat{X}, \subseteq)
            \arrow[r, "\oth_{\hat{\gamma}}" below]
        & \fun{iF}_2\Psi
    \end{tikzcd}
  $$
  where $\eta$ is the unit $\id_{\cat{MF}} \to \fun{FF}$.
  Thus two states $x, x' \in X$ are logically equivalent if and only if
  $\eta_X(x)$ and $\eta_X(x')$ are behaviourally equivalent when
  conceived of as states in $(\hat{X}, \subseteq, \hat{\gamma})$.

\section{A monotone modality}\label{sec:monotone}

  In this section we discuss the enrichment of $\mathbf{MI}$ with a
  monotone modal operator, that we denote by $\mon$.
  We present this dialgebraically, and leave explicit descriptions of
  the logical axioms, frame semantics (a variation of neighbourhood semantics)
  and descriptive frames to the reader.
  
  The semantics of this logic are given by a simple adaptation of
  monotone neighbourhood frames, see e.g.~\cite{Che80,Han03}.

\begin{defi}
  Let $(X, \leq)$ be an I-frame. Define
  $\fun{H}(X, \leq) = \{ W \subseteq \fun{F}(X, \leq) \mid
  \text{ if } a \in W \text{ and } a \subseteq b \text{ then } b \in W \}$.
  If we order $\fun{H}(X, \leq)$ by inclusion then it forms a semilattice
  with intersection as meet and the whole of $\fun{F}(X, \leq)$ as top element.
  For an I-frame morphism $h : (X, \leq) \to (X', \leq')$,
  define $\fun{H}h : \fun{H}(X, \leq) \to \fun{H}(X', \leq')$ by
  $\fun{H}h(W) = \{ a' \in \fun{F}(X', \leq') \mid h^{-1}(a') \in W \}$.
\end{defi}

  Intuitively, $(\fun{i}, \fun{H})$-dialgebras are I-frames with an additional
  map that assigns to each world $x$ an
  up-closed collection of neighbourhoods, and the neighbourhoods are
  filters of $(X, \leq)$.
  We now define our intended logic as a dialgebraic logic for
  $(\fun{i}, \fun{H})$-dialgebras.

\begin{defi}
  The predicate lifting $\lambda^{\mon} : \fun{UFi} \to \fun{UFH}$
  is given on components by
  $$
    \lambda^{\mon}_{(X, \leq)}(a) = \{ W \in \fun{H}(X, \leq) \mid a \in W \}.
  $$
  We abbreviate $\mon a = \heartsuit^{\lambda^{\mon}}a$, and work
  with a single axiom: $\mon (a \wedge b) \leq \mon a$.
\end{defi}

  It is obvious that $\lambda^{\mon}_{(X, \leq)}(a)$ is a filter for all I-frames
  $(X, \leq)$ and filters $a$. (It is nonempty because the collection of all
  filters is in it, it is up-closed because $a \in W \subseteq W'$ implies
  $a \in W'$, and it is closed under meets because meets are given by
  intersection and $a \in W$ and $a \in W'$ implies $a \in W \cap W'$.)

\begin{prop}
  The assignment $\lambda^{\mon}$ is a natural transformation.
\end{prop}
\begin{proof}
  Let $f : (X, \leq) \to (X', \leq')$ be an I-frame morphism.
  We need to show that
  $$
    \begin{tikzcd}
      \fun{UFi}(X, \leq)
            \arrow[r, "\lambda^{\mon}_{(X, \leq)}"]
        & [1em] \fun{UFH}(X, \leq) \\ [1em]
      \fun{UFi}(X', \leq')
            \arrow[r, "\lambda^{\mon}_{(X', \leq')}"]
            \arrow[u, "\fun{UFi}f"]
        & \fun{UFH}(X', \leq')
            \arrow[u, "\fun{UFH}f" swap]
    \end{tikzcd}
  $$
  commutes.
  To see this, let $p'$ be a filter of $(X', \leq')$ and $W \in \fun{H}(X, \leq)$
  and compute
  \begin{align*}
    W \in \fun{UFH}f(\lambda^{\mon}_{(X', \leq')}(p'))
      &\iff \fun{H}f(W) \in \lambda^{\mon}_{(X', \leq')}(p') \\
      &\iff p' \in \fun{H}f(W) \\
      &\iff f^{-1}(p') \in W \\
      &\iff W \in \lambda^{\mon}_{(X, \leq)}(\fun{UFi}f(p'))
  \end{align*}
  This proves that the diagram commutes.
\end{proof}

  The dialgebraic logic given by this predicate lifting and axiom
  is the pair $(\fun{M}, \rho)$.
  The functor $\fun{M} : \cat{ISL} \to \cat{SL}$ is given on objects
  by sending an implicative meet-semilattice $A$ to the free meet-semilattice
  generated by $\{ \mon a \mid a \in A \}$ modulo $\mon(a \wedge b) \leq \mon a$.
  The action of $\fun{M}$ on a morphism $h : A \to B$ is defined on generators
  by $\fun{M}h(\mon a) = \mon h(a)$.
  The natural transformation $\rho$ is given by
  $$
    \rho_{(X, \leq)}
      : \fun{MF}_1(X, \leq) \to \fun{FH}(X, \leq)
      : \mon a \mapsto \{ W \in \fun{H}(X, \leq) \mid a \in W \}.
  $$

\begin{prop}
  The assignment $\rho$ is well defined, hence a natural transformation.
\end{prop}
\begin{proof}
  Let $(X, \leq)$ be an I-frame and $a, b$ filters in it.
  We have $a \wedge b \leq a$ so for all $W \in \fun{H}(X, \leq)$,
  $a \wedge b \in W$ implies $a \in W$.
  Therefore
  $$
    \rho_{(X, \leq)}(\mon(a \wedge b))
      = \{ W \in \fun{H}(X, \leq) \mid a \wedge b \in W \}
      \subseteq \{ W \in \fun{H}(X, \leq) \mid a \in W \}
      = \rho_{(X, \leq)}(\mon a).
  $$
  So $\rho$ is well defined and the dialgebraic logic is sound.
\end{proof}

  As we have seen in Section \ref{sec:dialg}, in order to obtain completeness,
  filter extensions and expressivity-somewhere-else, it suffices to find a
  natural transformation $\tau : \fun{FM} \to \fun{HF}_2$
  such that $\rho^{\flat} \circ \tau = \id$.
  We complete this section by giving such a $\tau$.
  To get inspiration for the definition of $\tau$, let us first unravel the definition
  of $\rho^{\flat}$. This is given on components by the concatenation
  $$
    \begin{tikzcd}
      \fun{HF}_2A
            \arrow[r, "\eta_{\fun{HF}_2A}"]
        & \fun{FFHF}_2A
            \arrow[r, "\fun{F}\rho_{\fun{F}_2A}"]
        & \fun{FMF}_1\fun{F}_2A
            \arrow[r, "\fun{FM}\theta'_A"]
        & \fun{FM}A.
    \end{tikzcd}
  $$
  In particular, for $W \in \fun{HF}_2A$ we have
  \begin{align*}
    \mon a \in \rho_A^{\flat}(W)
      &\iff \mon\tilde{a} \in \fun{F}\rho_{\fun{F}_2A}(\eta_{\fun{HF}_2A}(W)) \\
      &\iff \rho_{\fun{F}_2A}(\mon\tilde{a}) \in \eta_{\fun{HF}_2A}(W) \\
      &\iff \{ V \in \fun{HF}_2A \mid \tilde{a} \in V \} \in \eta_{\fun{HF}_2A}(W) \\
      &\iff W \in \{ V \in \fun{HF}_2A \mid \tilde{a} \in V \} \\
      &\iff \tilde{a} \in W
  \end{align*}
  
  In our definition of $\tau$ we require the notions of a clopen
  and closed filter. We call a filter on $\fun{F}_2A$ (i.e.~an element
  of $\fun{FF}_2A$) \emph{clopen} if it is of the form
  $\tilde{a} = \{ x \in \fun{F}_2A \mid a \in x \}$
  and \emph{closed} if it is the intersection of all clopen filters in
  which it is contained (cf.~Definition \ref{def:fun-V}).
  
\begin{defi}
  Let $A$ be an implicative semilattice and $U \in \fun{FM}A$. Then we define
  $\tau_A(U) \in \fun{HF}_2A$ by:
  \begin{itemize}
    \item For all clopen filters $\tilde{a} \in \fun{FF}_2A$, let
          $\tilde{a} \in \tau_A(U)$ if $\mon a \in U$;
    \item For all closed filters $c$, let $c \in \tau_A(U)$ if all clopen filters
          containing $c$ are in $\tau_A(U)$;
    \item If $p$ is any filter, let $p \in \tau_A(U)$ if there is a closed filter $c$
          such that $c \subseteq p$ and $c \in \tau_A(U)$.
  \end{itemize}
  This yields an assignment $\tau : \fun{FM} \to \fun{HF}_2$.
\end{defi}

  It is trivial to see that $\tau_A$ is well defined for every implicative semilattice $A$.

\begin{lem}
  For each $A \in \cat{ISL}$, the map $\tau_A$ is an M-frame morphism.
\end{lem}
\begin{proof}
  First we prove that $\tau_A$ preserves top elements.
  The top element of $\fun{FM}A$ is $\fun{M}A$ and
  we observe that $\{ \top \} \in \tau_A(\fun{M}A)$ because
  $\{ \top \} = \bigcap \{ \tilde{a} \mid a \in A \}$.
  Furthermore this implies $\{ \top \}$ is closed. Since every other filter
  of $A$ is a superset of $\{ \top \}$ this implies that
  $\tau_A(\fun{M}A) = \fun{F}_2A$, which is the top element of
  $\fun{HF}_2A$.
  
  We now prove that $\tau_A$ preserves meets.
  Let $U, V \in \fun{FM}A$. We have to prove that
  $$
    \tau_A(U) \cap \tau_A(V) = \tau_A(U \cap V).
  $$ 
  If $p \in \tau_A(U \cap V)$ then
  there is a closed filter $c \subseteq p$ such that for all clopen filters
  $\tilde{a}$ containing $c$, $\mon a \in U \cap V$. Clearly this implies
  that $c \in \tau_A(U)$ and $c \in \tau_A(V)$, hence $p \in \tau_A(U) \cap \tau_A(V)$.
  Conversely, suppose $p \in \tau_A(U) \cap \tau_A(V)$. Then there are closed
  filters $c_U \in \tau_A(U)$ and $c_V \in \tau_A(V)$ that are both contained
  in $p$.
  Let $c = \bigcap \{ \tilde{a} \mid c_U \subseteq \tilde{a} \text{ and } c_V \subseteq \tilde{a} \}$.
  Then we know from the duality for the Vietoris filter functor that
  $c = c_U \wwedge c_V = \{ x \cap y \mid x \in c_U, y \in c_V \}$.
  But since both $c_U$ and $c_V$ are subsets of $p$ and $p$ is a filter 
  (hence closed under $\cap$, which acts as meet in $\fun{F}_2A$)
  this implies that $c \subseteq p$.
  By construction $c \in \tau_A(U \cap V)$,
  so as a consequence we have $p \in \tau_A(U \cap V)$.
\end{proof}

\begin{lem}
  $\tau$ is natural.
\end{lem}
\begin{proof}
  Let $f : A \to B$ be an implicative semilattice morphism. We need to show that
  $$
    \begin{tikzcd}
      \fun{FM}A
            \arrow[r, "\tau_A"]
        & \fun{HF}_2A \\
      \fun{FM}B
            \arrow[r, "\tau_B"]
            \arrow[u, "(\fun{M}f)^{-1}"]
        & \fun{HF}_2B
            \arrow[u, "\fun{H}f^{-1}" right]
    \end{tikzcd}
  $$
  commutes.
  Let $U \in \fun{FM}B$ and $p \in \fun{FF}_2A$.
  Since elements of $\fun{HF}_2A$ are uniquely determined by the clopen filters
  they contain, it suffices to prove that
  $\tilde{a} \in \tau_A((\fun{M}f)^{-1}(U))$ if and only if
  $\tilde{a} \in \fun{H}f^{-1}(\tau_B(U))$ for all $a \in A$.
  To see that this is indeed the case, compute
  \begin{align*}
    \tilde{a} \in \tau_A((\fun{M}f)^{-1}(U))
      &\iff \mon a \in (\fun{M}f)^{-1}(U) \\
      &\iff (\fun{M}f)(\mon a) \in U \\
      &\iff \mon f(a) \in U \\
      &\iff \widetilde{f(a)} \in \tau_B(U) \\
      &\iff (f^{-1})^{-1}(\tilde{a}) \in U \\
      &\iff \tilde{a} \in \fun{H}f^{-1}(\tau_B(U))
  \end{align*}
  So $\tau$ is a natural transformation.
\end{proof}

  We now show that $\tau$ is right inverse to $\rho^{\flat}$.

\begin{prop}
  For all $A \in \cat{ISL}$, the composition
  $$
    \begin{tikzcd}
      \fun{FM}A
            \arrow[r, "\tau_A"]
        & \fun{HF}_2A
            \arrow[r, "\rho_A^{\flat}"]
        & \fun{FM}A
    \end{tikzcd}
  $$
  is the identity.
\end{prop}
\begin{proof}
  We have seen that $\mon a \in \rho^{\flat}(W)$ iff $\tilde{a} \in W$.
  Since a filter in $\fun{FM}A$ is determined uniquely by the elements
  of the form $\mon a$ it contains (where $a \in A$), it suffices to
  show that for all $U \in \fun{FM}A$ we have
  \begin{equation}
    \mon a \in U \iff \tilde{a} \in \tau_A(U).
  \end{equation}
  The direction from left to right holds by definition.
  Conversely, suppose $\tilde{a} \in \tau_A(U)$.
  Then there exists a closed filter $c \in \tau_A(U)$ such that $c \subseteq \tilde{a}$.
  But this implies $\mon a \in U$ by definition of $\tau$.
\end{proof}

  As a consequence of the existence of such a $\tau$, we get 
  completeness, filter extensions and expressivity-somewhere-else.
  Concretely, the filter extension of an $(\fun{i}, \fun{H})$-dialgebra
  $(X, \leq, \gamma)$ is based on the implicative semilattice
  $(\hat{X}, \subseteq) = \fun{FF}(X, \leq)$.
  The dialgebra structure $\hat{\gamma}$ on $(\hat{X}, \subseteq)$ is given as follows.
  For a filter $p \in \hat{X}$ and a filter
  $U \in \fun{F}(\hat{X}, \subseteq)$ we define:
  \begin{itemize}
    \item If $U$ is a clopen filter, i.e.~if $U = \tilde{a}$ for some
          $a \in \fun{F}(X, \leq)$,
          then $U \in \hat{\gamma}(p)$ iff $\{ x \in X \mid a \in \gamma(x) \} \in p$;
    \item If $U$ is a closed filter, then $U \in \hat{\gamma}(p)$ iff all
          clopen filters containing $U$ are in $\hat{\gamma}(p)$;
    \item If $U$ is any filter, then $U \in \hat{\gamma}(p)$ iff there
          is a closed filter $C$ such that $C \subseteq U$ and $C \in \hat{\gamma}(p)$.
  \end{itemize}
  Expressivity-somewhere else then states:
  
\begin{thm}
  Two states $x$ and $y$
  in $(X, \leq, \gamma)$ are logically equivalent if and only if
  $\eta_{(X, \leq)}(x), \eta_{(X, \leq)}(y) \in (\hat{X}, \subseteq)$ are 
  behaviourally equivalent. That is, if they are identified by some
  $(\fun{i}, \fun{H})$-dialgebra morphism.
\end{thm}

  As explained in Section \ref{sec:dialg}, this theorem is a consequence
  of general results in \cite{GroPat20}.

\begin{rem}
  It can be proved in a similar manner as in Section \ref{sec:fragments} that the system
  $\mathbf{MI_{\mon}} = \mathbf{MI} + \mon(a \wedge b) \leq \mon a$
  characterises the meet-implication-modality-fragment of
  intuitionistic logic with a \emph{geometric modality} from \cite[Section 6]{Gol93}.
  The latter was called monotone modal intuitionistic logic in \cite[Section 8]{GroPat20}.
\end{rem}

\section{Conclusion}

  We have started investigating modal extensions of the $(\top, \wedge, \to)$-fragment
  of intuitionistic logic, and shown that this fits in the framework
  of dialgebraic logic.
  We list some potential directions for further research.

  \subsection*{Different modal operators}
        While we have given two modal extensions of $\mathbf{MI}$, there
        may be many other interesting ones. For example, it would be interesting
        to investigate modalities of higher arities,
        or a weak negation modality like in \cite{Cel99}.
  \subsection*{Kripke-style bisimulation}
        Another equivalence notion between $\Box$-frames, besides behavioural equivalence,
        is given by Kripke-style bisimulation. While it is clear that
        bisimilar states are logically equivalent, it would be interesting
        to investigate the converse (and thus derive a Hennessy-Milner result).
        Similar investigations can be carried out for the extension
        of $\mathbf{MI}$ with a monotone operator,
        where the notion of bisimulation should be inspired by the one for
        monotone modal logic over a classical base \cite{Han03,HanKup04}.
  \subsection*{Co- and contravariant functors}
        In many examples of modal logic based on a dual adjunction,
        the contravariant functor turning frames into algebras has a
        \emph{covariant} counterpart:
        For classical modal logic we have a co- and contravariant powerset functor;
        In positive modal logic, taking the collection of up-closed subsets of a poset
        gives rise to both a co- and a contravariant functor.
        The same happens here, where the contravariant filter functor $\fun{F}$
        and the covariant filter functor $\fun{G}$ (that can be extended to $\cat{MF}$)
        coincide on objects.
        This raises the question whether these are instances of a more
        general phenomenon. Perhaps one can identify those dual adjunction
        that allow for such a covariant counterpart.

%

\bibliographystyle{alphaurl}
\bibliography{meet-and-to-biblio.bib}{}

\appendix
\section{A Hilbert-style system for meet-implication logic}

  We present a Hilbert-style axiom system for the meet-implication
  fragment of intuitionistic propositional logic and establish
  soundness and completeness with respect to an equational
  axiomatisation.  We consider the $\lor$-free fragment of
  propositional logic over a set $V$ of propositional variables.
  More precisely, the language $\lan{L}$ is given by the grammar
  $$
    \phi ::= \top \mid p \mid \phi \wedge \phi \mid \phi \to \phi
  $$
  where $p \in \Prop$ ranges over the set of propositional variables. We
  use standard operator precedence and assume that $\land$ binds
  more tightly than $\to$.

\begin{defi}
  The Hilbert-style deductive system for the meet-implication
  fragment of intuitionistic propositional logic is given by the
  axioms
  \begin{enumerate}[(H$_1$)]
    \item \label{it:H1} $a \to (b \to a)$
    \item \label{it:H2} $(a \to (b \to c)) \to ((a \to b) \to (a \to c))$
    \item \label{it:H3} $(a \wedge b) \to a$
    \item \label{it:H4} $(a \wedge b) \to b$
    \item \label{it:H5} $a \to (b \to (a \wedge b))$
    \item \label{it:H6} $\top$
  \end{enumerate}
  If $\Gamma$ is a set of formulae and $a$ is a formula, we 
  say that $a$ \emph{is deducible in $\Hilb$} from $\Gamma$ 
  if $\Gamma \entails_\Hilb a$ can be derived using the
  rules 
  \[ \mbox{(\assum)}\frac{}{\Gamma \entails_\Hilb a} (\mbox{if } a \in \Gamma)
  \qquad\quad
     \mbox{(\axio)}\frac{}{\Gamma \entails_\Hilb a} (\mbox{if } a \in H) 
  \qquad\quad
   \mbox{(\mopo)}\frac{\Gamma \entails_\Hilb a \qquad \Gamma \entails_\Hilb
  a \to b}{\Gamma \entails_\Hilb b}
  \]
  where $H$ is the set of substitution instances of the axioms
  \ref{it:H1} -- \ref{it:H6} above.
  We write $\Hilb \entails a$ if $\emptyset \entails_\Hilb a$.
\end{defi}

  Our goal is to show that this Hilbert system is equivalent to
  the equational system used to axiomatise implicative meet
  semilattices in the literature (see e.g.~\cite{Nem65}) introduced next.

\begin{defi}
  Let $\ms{E}$ be the equational system consisting of
  \begin{enumerate}[(E$_1$)]
    \item \label{it:E1} $a \wedge (b \wedge c) = (a \wedge b) \wedge c$
    \item \label{it:E2} $a \wedge b = b \wedge a$
    \item \label{it:E3} $a \wedge a = a$
    \item \label{it:E4} $a \wedge \top = a$.
  \end{enumerate}
  The set of \emph{provable} equations in $\Eq$ is the least set of
  equations containing all substitution instances of \ref{it:E1}
  -- \ref{it:E4} above that is closed under the rules of equational
  reasoning
  \[ (\mathrm{ref})\frac{}{a = a} \qquad (\mathrm{sym})\frac{a = b}{b = a} \qquad
     (\mathrm{trans})\frac{a = b \quad b = c}{a = c} 
     \qquad(\mathrm{cong}) \frac{a_1 = b_1 \quad a_2 = b_2}{a_1 \circ a_2 = b_1
     \circ b_2}  
   \]
  for $\circ \in \lbrace \land, \to \rbrace$, as well as the
  residuation equivalence  (displayed as a pair of rules)
  \[  \frac{a \land b \leq c}{a \leq b \to c} 
      \qquad 
      \frac{a \leq b \to c}{a \land b \leq c} \]
  where we write, as usual, $a \leq b$ for $a \land b = a$.  
\end{defi}

Our goal is to establish soundness and completeness of $\Hilb$ with
respect to $\Eq$. That is, $\Hilb \entails a$ iff $\Eq \entails a =
\top$. We begin with a few simple facts on derivability in $\Hilb$,
where $\Gamma$ is a set of formulae, and $a, b$ are formulae. The
first is the admissibility of weakening.

\begin{lem}[Weakening]\label{lemma:weak}
  $\Gamma, a \entails_\Hilb  b$ whenever $\Gamma \entails_\Hilb b$.
\end{lem}

\noindent
The proof is a routine induction on derivations, and therefore
omitted. The next two facts are used in the proof of the deduction
theorem. 

\begin{lem}\label{lemma:simple}
  We have
  (1) $\Gamma \entails_\Hilb a \to a$;
  and (2) If $\Gamma \entails_\Hilb a$ then $\Gamma \entails_\Hilb b \to a$.
\end{lem}
\begin{proof}
For the second item, note that $\Gamma \entails_\Hilb a \to (b \to
a)$ by \ref{it:H1}. As $\Gamma \entails_\Hilb a$ by assumption, an
application of (mp) gives $\Gamma \entails_\Hilb b \to a$. 

For the first claim, note that we have the following three instances
of axioms
\begin{deduction}
1 & \Gamma \entails_\Hilb & (a \to (a \to a) \to a) \to ((a \to (a \to
a))  \to (a \to a)) & \ref{it:H2} \\
2 & \Gamma \entails_\Hilb & a \to ((a \to a) \to a) &
\ref{it:H1} \\
3 & \Gamma \entails_\Hilb & a \to (a \to a) & \ref{it:H1}
\end{deduction}
whence we may conclude that
\begin{deduction}
4 & \Gamma \entails_\Hilb & (a \to (a \to a))  \to (a \to a) &
(\mopo 2 1) \\
5 & \Gamma \entails_\Hilb & a \to a & (\mopo 3 4)
\end{deduction}
as required.
\end{proof}

With these preparations, we are now ready for the deduction theorem.

\begin{lem}[Deduction Theorem]\label{lemma:ded-thm}
  We have $\Gamma \entails_\Hilb a \to b$ if and only if $\Gamma, a
  \entails_\Hilb b$.
\end{lem}
\begin{proof}
  For ``only if'' assume that $\Gamma \entails_\Hilb a \to b$. Using
  weakening (Lemma \ref{lemma:weak}) we obtain that $\Gamma, a
  \entails_\Hilb a \to b$. As $\Gamma, a \entails_\Hilb a$ by
  (\assum), an application of (\mopo) yields the claim.
  
  We show ``if'' by induction on the derivation of $\Gamma, a \vdash_\Hilb b$.
  If $\Gamma, a \entails_\Hilb b$ has been derived using (\assum),
  we distinguish the cases $b \in \Gamma$ and $a = b$. 
  In the first case, 
  $\Gamma \entails_\Hilb b$ (again using (\assum)) 
  whence $\Gamma \entails_\Hilb a \to b$
  using Lemma \ref{lemma:simple}. If $a = b$, then 
  Lemma \ref{lemma:simple} again gives $\Gamma \entails a \to a$,
  i.e.~$\Gamma \entails a \to b$.

  If $b \in H$ is a substitution
  instance of \ref{it:H1} -- \ref{it:H6}, then $\Gamma \entails b$
  (for the same reason), and again Lemma \ref{lemma:simple} gives
  $\Gamma \entails a \to b$.
  
  The final case to consider is that 
  $\Gamma, b$ was derived using
  (\mopo). This means that there is a formula $c$ and shorter
  derivations of $\Gamma, a \entails_\Hilb c \to b$ and $\Gamma, a
  \entails_\Hilb c$. We apply the induction hypotheses and
  instantiate \ref{it:H2} to obtain
  \begin{deduction}
  1 & \Gamma \entails_\Hilb & a \to (c \to b) &(IH)\\
  2 & \Gamma \entails_\Hilb & a \to c &(IH) \\
  3 & \Gamma \entails_\Hilb & (a \to (c \to b)) \to ( (a \to c) \to (a
  \to b)) &\ref{it:H2}
  \end{deduction}
  which allows us to use $(\mopo)$ twice to obtain that
  \begin{deduction}
  4 & \Gamma \entails_\Hilb & (a \to c) \to (a \to b)  & (\mopo 1 3) \\
  5 & \Gamma \entails_\Hilb & a \to b & (\mopo 2 4)
  \end{deduction}
  as required, which finishes the proof.
\end{proof}

%

Our reasoning so far is valid in minimal logic, and we are adding
conjunction. The next lemma describes the mechanics of reasoning
with conjunction, both on the left and on the right.

\begin{lem}\label{lemma:conj}
  $\Gamma \entails a \land b$ if and only if both $\Gamma \entails a$ and
  $\Gamma \entails_\Hilb b$. 
  Also, for any formula $c$, 
  $\Gamma, a \land b \entails c$ if and only if $\Gamma, a, b
\entails_\Hilb c$.
\end{lem}
\begin{proof}
  The first claim is straightforward. Note that $\Gamma \entails a \to (b \to a \land b)$ by (\axio). Now
  apply (\mopo) twice. For the converse, we have that both $\Gamma
  \entails a \land b \to a$ and $\Gamma \entails a \land b \to b$
  and we obtain the claim using (\mopo). We turn to the second
  statement.

  First suppose that $\Gamma, a \land b \entails c$. By the
  deduction theorem, $\Gamma \entails_\Hilb a \land b \to c$. Using
  weakening this gives $\Gamma, a, b \entails_\Hilb a \land b \to
  c$. On the other hand, we have $\Gamma, a, b \entails_\Hilb a$ and
  $\Gamma, a, b \entails_\Hilb b$ using (\assum) so that the first part of
  this lemma gives $\Gamma, a, b \entails_\Hilb a \land b$.
  We then get $\Gamma, a, b \entails_\Hilb c$ from (\mopo).

  Conversely, suppose that $\Gamma, a, b \entails_\Hilb c$. 
  It is easy to see that $\Gamma, a \land b \entails_\Hilb a$, as
  $\Gamma, a \land b \entails_\Hilb a \land b$ by (\assum). Also,
  $\Gamma, a \land b \entails_\Hilb a \land b \to a$ by (\axio) so
  that applying (\mopo) gives $\Gamma, a \land b \entails a$. A
  symmetric argument shows that $\Gamma, a \land b \entails_\Hilb
  b$. 
  Using the deduction theorem and weakening, we obtain $\Gamma,
  a \land b \entails_\Hilb a \to (b \to c)$ from our assumption
  $\Gamma, a, b \entails_\Hilb c$. As both $\Gamma, a \land b
  \entails_\Hilb a$ and $\Gamma, a \land b \entails_\Hilb b$, applying (\mopo)
  twice yields $\Gamma, a \land b \entails_\Hilb c$ as desired.
\end{proof}

%
%
We now have collected all required prerequisites for the soundness
of the equational system $\Eq$ with respect to $\Hilb$. 
\begin{prop}
  If $\ms{E} \vdash a = b$, then $\ms{H} \vdash a \to b$ and $\ms{H} \vdash b \to a$.
\end{prop}
\begin{proof}
  The proof proceeds by induction on the derivation.
  We first show that $\Hilb \entails
  a \to b$ and $\Hilb \entails b \to a$ for all substitution
  instances of the equations \ref{it:E1} -- \ref{it:E4}. 

  We begin with \ref{it:E1}, i.e.~associativity. By Lemma
  \ref{lemma:conj} we have that $a, b, c \entails_\Hilb a \land b$,
  and a second application yields $a, b, c \entails_\Hilb (a \land
  b) \land c$. Applying the second part of the same lemma twice now gives
  $a \land (b \land c) \entails_\Hilb (a \land b) \land c$. A
  symmetric argument shows that $\Hilb \entails (a \land b) \land c
  \to a \land (b \land c)$. 

  For commutativity, i.e.~\ref{it:E2}, note that $a, b \entails b
  \land a$ by Lemma \ref{lemma:conj} whence $a \land b
  \entails_\Hilb b \land a$ by the same lemma. The deduction theorem
  gives $\Hilb \entails a \land b \to b \land a$. Swapping $a$ and
  $b$ gives the converse implication.

  We turn to \ref{it:E3}, idempotency. Note that $\Hilb \entails a
  \land a \to a$ by (\axio). The converse direction, i.e. $\Hilb
  \entails a \to a \land a$ follows from observing that $a \to a
  \land a$ is a consequence of Lemma \ref{lemma:conj} and an
  application of the deduction theorem. 

  For identity, that is \ref{it:E4},  note that $a \land \top
  \to a$ is an instance of \ref{it:H3}. We show the converse
  direction, i.e.~$\Hilb \entails a \to a \land \top$. As $a
  \entails_\Hilb a$ and $a \entails_\Hilb \top$ by (\assum) and (\axio),
  respectively, Lemma \ref{lemma:conj} gives that $a \entails_\Hilb
  a \land \top$, and an application of the deduction theorem yields
  the claim.

  We turn to the laws of equational logic. For (\refl), note that
  $\Hilb \entails a \to a$ by Lemma \ref{lemma:simple}. If $\Eq
  \entails a = b$ is has been derived from $\Eq \entails b = a$ using
  symmetry, the claim is immediate as the inductive hypothesis
  yields $\Hilb \entails b \to a$ and $\Hilb \entails a \to b$. 

  Now assume that $\Eq \entails a_1 \land b_1 = a_2 \land b_2$ has
  been derived using congruence from $\Eq \entails a_1 = a_2$ and
  $\Eq \entails b_1 = b_2$. We only show that $\Hilb \entails a_1
  \land b_1 \to a_2 \land b_2$ as the other implication is almost
  identical. Applying the deduction theorem to the induction
  hypotheses (the shorter derivations of $a_1 = a_2$ and $b_1 =
  b_2$) we obtain that $a_1 \entails_\Hilb a_2$ and $b_1
  \entails_\Hilb b_2$. Using weakening and Lemma \ref{lemma:conj}
  this gives $a_1, b_1 \entails a_2 \land b_2$ and applying Lemma
  \ref{lemma:conj} again, together with the deduction theorem,  gives
  $\Hilb \entails a_1 \land b_1  \to a_2 \land b_2$.

  For the second congruence law, assume that $\Eq \entails  a_1 \to
  b_1 = a_2 \to b_2$ has 
  been derived using congruence from $\Eq \entails a_1 = a_2$ and 
  $\Eq \entails b_1 = b_2$. Again, we only demonstrate that $\Hilb
  \entails (a_1 \to b_1) \to (a_2 \to b_2)$.  The induction
  hypotheses for the derivations of $a_1 = a_2$ and $b_1 = b_2$
  imply
  $\Hilb \entails a_2 \to a_1$ and $\Hilb \entails b_1 \to b_2$,
  respectively. We obtain:
  \begin{wideded}
    1 & a_1 \to b_1, a_2 \entails_\Hilb & a_2 \to a_1 & 
        (IH) and weakening \\
    2 & a_1 \to b_1, a_2 \entails_\Hilb & a_1 & 
        (\mopo) using (\assum) and (1)\\
    3 & a_1 \to b_1, a_2 \entails_\Hilb & b_1 &
        (\mopo) using (2) and (\assum) \\
    4 & a_1 \to b_1, a_2 \entails_\Hilb & b_1 \to b_2 &
        (IH) and weakening \\
    5 & a_1 \to b_1, a_2 \entails_\Hilb & b_2 &
        (\mopo) using (3) and (4)
  \end{wideded}
  so that a double application of the deduction theorem yields the
  claim, i.e.~$\Hilb \entails (a_1 \to b_1) \to (a_2 \to b_2)$.

  This leaves the residuation rule. First assume that $\Eq \entails a \leq b \to
  c$ has been concluded from $\Eq \entails a \land b \leq c$. Unfolding
  the definition of $\leq$, we assume that $\Eq \entails (a \land b)
  \land c = a \land b$ and show that both $\Hilb \entails (a \land (b \to c))
  \to a$ and $\Hilb \entails a \to (a \land b \to c)$. The former is
  a substitution instance of (\axio). For the latter, note that the
  inductive hypothesis for the shorter derivation $\Eq \entails a
  \land b  = (a \land b) \land c$ implies that $\Hilb \entails (a
  \land b) \to (a \land b) \land c$. The deduction theorem and Lemma
  \ref{lemma:conj} gives $a, b \entails_\Hilb (a \land b) \land c$.
  Another application of the same lemma, together with the deduction
  theorem yields $a\entails_\Hilb b \to c$. Again, using Lemma
  \ref{lemma:conj} and (\assum), we obtain $a \entails_\Hilb a \land
  (b  \to c)$. Applying the deduction theorem finally yields $\Hilb
  \entails a \to a \land (b \to c)$ as required.

  Now assume that $\Eq \entails a \land b \leq c$ has been derived
  from $\Eq \entails a \leq b \to c$. Again unfolding $\leq$, we
  assume that $\Eq \entails a \land (b \to c) = a$ and have to show
  that $\Hilb \entails (a \land b) \land c \to a \land b$ and
  $\Hilb \entails a \land b \to (a \land b) \land c$. The former is
  a substitution instance of axiom \ref{it:H3}.  For the latter,
  note that the inductive hypothesis, coupled with the deduction
  theorem, gives $a \entails_\Hilb a \land b \to c$.  Applying Lemma
  \ref{lemma:conj} and the deduction theorem a second time gives $a,
  b \entails c$. Now we apply Lemma \ref{lemma:conj} again to obtain
  $a \land b \entails_\Hilb c$. Using (\assum) and Lemma
  \ref{lemma:conj} yet again, this entails $a \land b \entails (a
  \land b) \land c$. A final application of the deduction theorem
  yields the claim.

This finishes the analysis of all cases in the derivation, and hence
the proof. 
\end{proof}

The first part of the promised equivalence of the systems $\Eq$ and
$\Hilb$ is now an easy corollary.

\begin{cor}\label{cor:eq-to-hilb}
  Suppose that $\Eq \vdash a = \top$. Then $\ms{H} \vdash a$.
\end{cor}

\begin{proof}
Using the previous lemma, we obtain that $\Hilb \entails \top \to
a$. As $\Hilb \entails \top$ by \ref{it:H6}, the claim follows from (\mopo).
\end{proof}

  Our next goal is the converse of the above corollary. That is, we will prove that
  $\ms{H} \vdash f$ implies $\ms{E} \vdash f = \top$. As before, we
  split the proof into several lemmas. Throughout, and as before, we write $a \leq
  b$ as a shorthand for $a \land b = a$.

\begin{lem}\label{lem:leqgeq}
  We have $\ms{E} \vdash a = b$ if and only if
  $\ms{E} \vdash a \leq b$ and $\ms{E} \vdash b \leq a$. Also, we
  always have that  $\Eq
  \entails a \leq \top$.
\end{lem}
\begin{proof}
  Given that $\Eq \entails a \leq b$ and $\Eq \entails b \leq a$, we
  have that $a = a \land b$ by unfolding $a \leq b$, and $b = b
  \land a$ by unfolding $b \leq a$. Hence $\Eq \entails a = b$ using
  symmetry and transitivity. If conversely $\Eq \entails a = b$, we
  have that $a \land b = a \land a = a$ using congruence,
  reflexivity, $(E_3)$ and transitivity. Similarly, one shows that
  $\Eq \entails b \land a = b$. The second statement, $\Eq \entails
  a \leq \top$, unfolds to $\Eq \entails a \land \top = a$ which is
  $(E_4)$.
\end{proof}


%
\begin{lem}\label{lem:res3}
  The following are deducible in $\Eq$:
  \begin{enumerate}
    \item $\ms{E} \vdash a \wedge b = (a \wedge b) \wedge (a \to b)$;
    \item $\ms{E} \vdash a \wedge b = a \wedge (a \to b)$.
  \end{enumerate}
\end{lem}
\begin{proof}
  For the first item,  note that $(a \land b) \land (a \land b)= a
  \land b$ is an instance of \ref{it:E3}. Using associativity
  \ref{it:E1} and congruence, this gives
  $((a \land b) \land a) \land b = a \land b$. Given that $(a \land
  b)  \land a  = a \land b$ is easily established using
  \ref{it:E1}, \ref{it:E2} and \ref{it:E3}, this yields $((a
  \land b) \land a) \land b = (a \land b) \land a$.
  By residuation, this yields $(a \land b) \land (a \to b) = a \land
  b$.
  For the second item, it suffices to show that 
  $\ms{E} \vdash a \wedge (a \to b) = (a \wedge b) \wedge (a \to b)$
  by transitivity. 
  It follows from symmetry and Lemma \ref{lem:leqgeq} that
  $a \to b \leq a \to b$ and residuation then yields
  $(a \to b) \wedge a \leq b$.
  Spelling out $\leq$ now yields
  $((a \to b) \wedge a) \wedge b = (a \to b) \wedge a$
  and the statement now follows from using \ref{it:E1} and \ref{it:E2}.
\end{proof}

  We can now establish the converse of Corollary \ref{cor:eq-to-hilb}.

\begin{lem}
  If $\ms{H} \vdash f$, then $\ms{E} \vdash f = \top$.
\end{lem}
\begin{proof}
  We proceed by 
  induction on the derivation of $\ms{H} \vdash a$, and begin with
  the axioms.
          Suppose $f = a \to (b \to a)$ is an instance of \ref{it:H1}.
          As a consequence of Lemma \ref{lem:leqgeq}
          it suffices to show that $\ms{E} \vdash \top \leq a \to (b \to a)$.
          By residuation this follows from $\top \wedge a \leq b \to a$,
          which in turn follows from $(\top \wedge a) \wedge b \leq a$.
          The last item follows from observing that
          \begin{align*}
            (\top \wedge a) \wedge b
              &= (a \wedge \top) \wedge b
              = a \wedge b
              = b \wedge a
              = b \wedge (a \wedge a) \\
              &= (b \wedge a) \wedge a
              = (a \wedge b) \wedge a
              = ((a \wedge \top) \wedge b) \wedge a
              = ((\top \wedge a) \wedge b) \wedge a
          \end{align*}
          is a valid chain of equalities in $\Eq$.
          
          Now suppose that $f = (a \to (b \to c)) \to ((a \to b) \to
          (a \to c))$ is an instance of \ref{it:H2}. Again, we
          observe that 
          \begin{align*}
            (a \to (b \to c)) &\wedge (a \to b) \wedge a \\
              &= (a \wedge (a \to (b \to c))) \wedge (a \to b) \\
              &= (a \wedge (b \to c)) \wedge (a \to b)  \\
              &= (a \wedge (a \to b)) \wedge (b \to c) \\
              &= (a \wedge b) \wedge (b \to c) \\
              &= a \wedge (b \wedge (b \to c)) \\
              &= a \wedge (b \wedge c) \\
          \end{align*}
          is a valid chain of equalities in $\Eq$, where we have
          used associativity and commutativity, as well as Lemma
          \ref{lem:res3}.
          Similar to the previous case, it now suffices to prove
          $$
            \ms{E} \vdash \top \leq (a \to (b \to c)) \to ((a \to b)
            \to (a \to c)).
          $$
          To see this, note that $a \wedge (b \wedge c) \leq c$ is
          clearly derivable in $\Eq$.  By congruence, we obtain 
          $(a \to (b \to c)) \wedge (a \to b) \wedge a \leq c$ which
          yields
          $(a \to (b \to c)) \wedge (a \to b) \leq a \to c$ by
          residuation.  Applying residuation again gives
          $a \to (b \to c) \leq (a \to b) \to (a \to c)$. Combining
          congruence and \ref{it:E4} we obtain
          $\top \wedge (a \to (b \to c)) \leq (a \to b) \to (a \to c)$
          from which we conclude that 
          $\top \leq (a \to (b \to c)) \to ((a \to b) \to (a \to c))$
          using residuation one more time. This completes the case
          of \ref{it:H2}.
          
          Now suppose $f = (a \land b) \to a$, i.e. $\Hilb \entails
          f$ is established as a substitution instance of
          \ref{it:H3}.
          It suffices to show that $\top \leq (a \wedge b) \to a$,
          and in turn, this is established once we show that
          $\top \wedge (a \wedge b) \leq a$. The latter unfolds to 
          $(\top \wedge (a \wedge b)) \wedge a = \top \wedge (a \wedge b)$.
          This follows from commutativity, identity and symmetry.

          If $f = (a \land b) \to b$, i.e. $\Hilb \entails f$ is a
          substitution instance of \ref{it:H4}, the argument is
          analogous to the previous case.
          
          Now suppose that $\Hilb \entails f$ was obtained as $f$ is
          a substitution instance of \ref{it:H5}, that is, $f = a
          \to (b \to (a \land b))$. 
          It suffices to show that $\top \leq a \to (b \to (a \wedge b))$.
          By symmetry we have $a \wedge b = a \wedge b$
          and it follows from \ref{it:E4} and commutativity that
          $(\top \wedge a) \wedge b = a \wedge b$.
          This in turn implies $\top \wedge (a \wedge b) \leq a \wedge b$
          so that by residuation we get $\top \wedge a \leq b \to (a \wedge b)$.
          Applying residuation again gives
          $\top \leq a \to (b \to (a \wedge b))$ as required.
          
          The last case of substitution instances that we need to
          consider is the case $f = \top$, i.e. instances of
          \ref{it:H6}.
          We need to show that $\ms{E} \vdash \top = \top$, which is
          (an instance of) the equational axiom of reflexivity.

         The last missing item of the proof is to consider $\Hilb
         \entails f$ to be derived using modus ponens. So suppose we
         have $\Hilb \entails e \to f$ and $\Hilb \entails e$. By
         induction hypothesis, we have that
         $\Eq \entails e \to f = \top$ and $\Eq \entails e = \top$. 
         For the claim it suffices to show that $\Eq \entails \top
         \leq f$ by Lemma \ref{lem:leqgeq}.
         From the inductive hypothesis, we obtain $\top  = e \to f$
         by symmetry, and, using the second inductive hypothesis $e
         = \top$, also $\top = \top \to f$ by congruence.
         Lemma \ref{lem:leqgeq} then gives $\top \leq \top \to f$
         which gives $\top \land \top \leq f$ by residuation. Using
         \ref{it:E3} and congruence, this finally yields $\top \leq
         f$ as required.
\end{proof}

\end{document}